\documentclass{article}

\RequirePackage[OT1]{fontenc}
\RequirePackage{amsthm,amsmath}

\usepackage[english]{babel}
\usepackage{amsmath}
\usepackage[margin=1.5in]{geometry}
\usepackage{amssymb,amsthm,bbm}
\usepackage[pdftex]{graphicx}
\usepackage{ifthen}
\usepackage{enumerate}
\usepackage{color}
\usepackage{dsfont}
\usepackage{algorithm,algorithmic}
\usepackage{subfig}
\usepackage[numbers]{natbib}

\theoremstyle{plain} \newtheorem{theorem}{Theorem}[section]
\theoremstyle{plain} \newtheorem{proposition}[theorem]{Proposition}
\theoremstyle{plain} \newtheorem{corollary}[theorem]{Corollary}
\theoremstyle{plain} \newtheorem{lemma}[theorem]{Lemma}
\theoremstyle{plain} 
\theoremstyle{remark} \newtheorem{remark}[theorem]{Remark}
\theoremstyle{plain}  
\theoremstyle{remark}\newtheorem{rem}{Remark}[section]

\newcounter{hypH}
\newenvironment{hypH}{\refstepcounter{hypH}\begin{itemize}
\item[{\bf H\arabic{hypH}}]}{\end{itemize}}

\newcounter{hypCov}

\def\Rset{\mathbb{R}}
\def\Nset{\mathbb{N}}

\newcommand{\esssup}{\operatornamewithlimits{ess \ sup}} 

\def\PP{\mathbb{P}}

\def\bfx{\mathbf{x}}

\def\bfy{\mathbf{y}}
\def\bfY{\mathbf{Y}}
\def\bfz{\mathbf{z}}

\def\bff{\mathbf{f}}

\def\denseps{\varphi}

\def\Esp{\mathbb{E}}
\def\Pr{\mathbb{P}}

\def\PPim{\PP_{\star}}

\newcommand{\1}{\ensuremath{\mathbf{1}}}
\newcommand{\rmd}{\mathrm{d}}

\newcommand{\rme}{\mathrm{e}}
\newcommand{\rmL}{\mathrm{L}}
\newcommand{\dimY}{\ell} 
\newcommand{\dimK}{m}
\newcommand{\norm}[1]{\|#1\|}

\newcommand{\norminf}[2]{\|#1\|_{#2}}
\newcommand{\eqsp}{\;}
\newcommand{\eqdef}{\ensuremath{\stackrel{\mathrm{def}}{=}}}
\newcommand{\paramstar}{f_{\star}}
\newcommand{\dens}{\nu}
\newcommand{\param}[2]{\widehat{f}_{#1}^{#2}}
\newcommand{\argmax}{\operatornamewithlimits{argmax}} 
 
\begin{document}

 \title{Nonparametric regression on hidden $\Phi$-mixing variables: identifiability and consistency of a pseudo-likelihood based estimation procedure}
 \author{Thierry Dumont\,\footnote{MODAL'X, Universit\'e Paris-Ouest, Nanterre, France. thierry.dumont@u-paris10.fr} and Sylvain Le Corff\,\footnote{Laboratoire de Math\'ematiques, Universit\'e Paris-Sud and CNRS, Orsay, France. sylvain.lecorff@math.u-psud.fr}}\,

\maketitle    

\begin{abstract}
This paper outlines a new nonparametric estimation procedure for unobserved $\Phi$-mixing processes. It is assumed that the only information on the stationary hidden states $(X_{k})_{k\ge 0}$ is given by the process $(Y_{k})_{k\ge 0}$, where $Y_{k}$ is a noisy observation of $\paramstar(X_{k})$. The paper introduces a maximum pseudo-likelihood procedure to estimate the function $\paramstar$ and the distribution $\dens_{b,\star}$ of $(X_0,\dots,X_{b-1})$ using blocks of observations of length $b$. The identifiability of the model is studied in the particular cases $b=1$ and $b=2$ and the consistency of the estimators of $\paramstar$ and of $\dens_{b,\star}$ as the number of observations grows to infinity is established. 
\end{abstract}

\section{Introduction}
\label{NPHMM:sec:intro}
The model considered in this paper consists of a bivariate stochastic process $\{(X_k,Y_k)\}_{k\geq 0}$ where only the sequence $(Y_k)_{k\ge 0}$ is observed. These observations are given by 
\begin{equation}
\label{NPHMM:eq:model}
Y_{k} = \paramstar(X_{k}) + \epsilon_{k}\eqsp,
\end{equation}
where $\paramstar$ is a function defined on a space $\mathbb{X}$ and taking values in $\Rset^{\dimY}$. The measurement noise $(\epsilon_{k})_{k\ge 0}$ is an independent and identically distributed (i.i.d.) sequence of Gaussian random vectors of $\Rset^{\dimY}$. This paper proposes a new method to estimate the function $\paramstar$ and the distribution of the hidden states using only the observations $(Y_k)_{k\geq 0}$. Nonparametric estimation with latent random variables is a challenging task and most of the existing results in this context use additional assumptions on the sequence $(X_k)_{k\geq 0}$. For instance, in \textit{errors-in-variables} models, the random variables $(X_k)_{k\ge 0}$ are observed through a sequence $(Z_k)_{k\ge 0}$, i.e. $Z_{k} = X_{k} + \eta_{k}$ and $Y_{k} = \paramstar(X_{k}) + \epsilon_{k}$, where the variables $(\eta_k)_{k\ge 0}$ are i.i.d with known distribution. 
Many solutions have been proposed to solve this problem, see \cite{fan:truong:1993} and \cite{ioannides:alevizos:1997} for a ratio of deconvolution kernel estimators, \cite{koo:lee:1998} for B-splines estimators and \cite{comte:taupin:2007} for a procedure based on the minimization of a penalized contrast. In the case where the hidden state is a Markov chain, \cite{lacour:2008a,lacour:2008b} considered the following observation model $Y_{k}= X_{k}+\epsilon_{k}$, where the random variables $\{\epsilon_{k}\}_{k\ge 0}$ are i.i.d. with known distribution.  \cite{lacour:2008a} (resp. \cite{lacour:2008b}) proposed an estimator of the transition density (resp. the stationary density and the transition density) of the Markov chain $(X_{k})_{k\ge 0}$ based on the minimization of a penalized $\rmL_2$ contrast. 

Recently, \cite{dumont:lecorff:2014} used the model \eqref{NPHMM:eq:model} for indoor simultaneous localization and mapping based on WiFi signals. In this framework, the process $(X_k)_{k \ge 0}$ is the position of a mobile device evolving in a building and it is assumed to be a Markov chain with transition density depending only on the distance between two consecutive states. $Y_k$ denotes the signal strengths measured by the device at time step $k$. Only $(Y_k)_{k\ge 0}$ is observed and inference on the hidden positions $(X_k)_{k\ge 0}$ (localization) requires an efficient estimation of $\paramstar$ (mapping). 

In this paper, the random process $(X_k)_{k\ge 0}$ is assumed to be $\Phi$-mixing and stationary which encompasses the i.i.d. case and the hidden Markov model setting of \cite{dumont:lecorff:2014}. We propose a new approach to estimate the function $\paramstar$ and the distribution $\dens_{b,\star}$ of the hidden states $(X_0,\dots,X_{b-1})$ for a given $b$ using only the observations $(Y_k)_{k\geq 0}$.  The identifiability of the model is studied and we show that for some particular cases, $f_\star$ may be recovered up to an isometric transformation of $\mathbb{X}$. The observations are decomposed into non-overlapping blocks $(Y_{kb},\dots,Y_{(k+1)b-1})$ to define a pseudo likelihood function. The estimator $(\widehat{f}_n,\widehat{\nu}_n)$ of $(\paramstar, \nu_{b,\star})$  is defined as a maximizer of a penalized version of the pseudo-likelihood of the observations $(Y_0,\dots,Y_{nb-1})$ over a class of functions $\mathcal{F}$ and a class of densities $\mathcal{D}_b$ on $\mathbb{X}^b$. These estimators of $\paramstar$ and $\dens_{b,\star}$ may then be used to define an estimator $\widehat{p}_{n}$ of the density of the distribution of $(Y_{0},\dots,Y_{b-1})$. It is proved in Section~\ref{NPHMM:sec:consistency} that the Hellinger distance between $\widehat{p}_{n}$ and the true distribution of a block of observations vanishes as the number of observations grows to infinity. This result is established using few assumptions on the model: the penalization function needs only to be lower bounded by a power of the supremum norm and no topological restrictions are made on $\mathbb{X}$. Under compacity assumptions on $\mathcal{F}$ and $\mathcal{D}_b$ the consistency of $(\param{n}{},\widehat{\dens}_{n})$ is derived although the rate of convergence of  $(\param{n}{},\widehat{\dens}_{n})$  remains an open problem and seems to be  very challenging. 

In Section~\ref{NPHMM:sec:identifiability}, we discuss the identifiability issues raised by the model \eqref{NPHMM:eq:model}. When $b=1$, the identifiability is studied in the particular case where  $\mathbb{X}$ is a subset of $\mathbb{R}^m$ for some $m>0$, $\paramstar$ is a $\mathcal{C}^1$ diffeomorphism and  $\mathcal{F}$ is a subset of  continuously differentiable functions on $\mathbb{X}$. We establish that if $\tilde{X}_0$ has a distribution with probability density $\dens$ and if $\tilde{f} \in \mathcal{F}$ is such that $\tilde{f}(\tilde{X}_0)$ and $\paramstar(X_0)$ have the same distribution then $\tilde{f}=\paramstar \circ \phi$ and $\nu = |J_\phi|\cdot\nu_{1,\star}\circ \phi$ where $\phi:\mathbb{X}\to\mathbb{X}$ is a bijective function ($|J_\phi|$ denotes the determinant of the Jacobian  matrix of $\phi$). This result only requires regularity assumptions on the unknown function $\paramstar$ and not on the candidate function $\tilde{f}$, which in particular is not assumed to be one-to-one. This implies that the inference task may be performed within a larger class of functions. A similar result is obtained when $b=2$  to establish that the model is identifiable up to an isometric transformation of $\mathbb{X}$ in the context of \cite{dumont:lecorff:2014}.
 
The consistency and identifiability results are applied in Section~\ref{NPHMM:sec:sobo} when $\mathcal{F}$ is assumed to be a Sobolev class of functions. In this setting, the supremum norm in $\mathcal{F}$ may be controlled by the penalty term to ensure that $\widehat{p}_n$ is consistent. Moreover, this framework satisfies the compacity assumption needed in Section \ref{NPHMM:sec:consistency} to derive the consistency of $(\widehat{f}_n,\widehat{\nu}_n)$.  Section~\ref{NPHMM:sec:sobo:simu} provides numerical experiments to illustrate our estimation approach and the identifiability results of Section~\ref{NPHMM:sec:identifiability}.
Proofs and technical results are postponed to Section~\ref{NPHMM:sec:proofs} and to the appendices. 

\section{Model and definitions}
\label{NPHMM:sec:model}
Let $(\Omega,\mathcal{E},\mathbb{P})$ be a probability space and  $(\mathbb{X},\mathcal{X})$ be a general state-space endowed with a measure $\mu$. Let $(X_k)_{k\ge 0}$ be a stationary process defined on $\Omega$ and taking values in $\mathbb{X}$. This process is only partially observed through the sequence $(Y_k)_{k\ge 0}$ which takes values in $\mathbb{R}^\dimY$, $\dimY\ge 1$.  In the sequel, for any $0\le k \le k'$, the sequence $(x_k,\dots,x_k')$ is written $x_{k:k'}$. The observations $(Y_k)_{k\ge 0}$ are given by \eqref{NPHMM:eq:model} where $\paramstar:\mathbb{X} \rightarrow \mathbb{R}^\dimY$ is a measurable function and the random variables $(\epsilon_{k})_{k\ge 0}$ are i.i.d. with density $\denseps$ with respect to the Lebesgue measure $\lambda$ of $\mathbb{R}^\ell$, given, for any $z_{1:\dimY}\in\mathbb{R}^{\dimY}$, by:
\begin{equation}
\label{NPHMM:eq:denseps}
\denseps(z_{1:\dimY}) \eqdef \left(2\pi\right)^{-\dimY/2}\exp\left( -\frac{1}{2}\sum_{j=1}^{\dimY} z_j^2\right)\eqsp.
\end{equation} 
In this paper, $\epsilon_0$ is assumed to be distributed according to a standard normal  distribution. Note that this setting is enough to deal with a known and nonsingular covariance matrix $\Sigma$. In this case, $(Y_k)_{k\ge 0}$ may be replaced by $(\Sigma^{-1/2}Y_k)_{k\ge 0}$ and the modified noise $\Sigma^{-1/2}\epsilon_0$ is then a standard normal random vector. 

This paper proposes a method to estimate the target function $\paramstar\in\mathcal{F}$, where $\mathcal{F}$ is a set of functions from $\mathbb{X}$ to $\mathbb{R}^{\dimY}$, and the distribution of the hidden states using only the observations $(Y_k)_{k\geq 0}$. This problem could be interpreted as a deconvolution problem where it is usual to assume that the noise distribution is known, see for instance \cite{carroll:hall:1988,Koo::1999,lacour:2006}. Here, the density $\varphi$ is assumed to be known to simplify the proof of identifiability (Section~\ref{NPHMM:sec:identifiability}). This proof only needs  the characteristic function of $\epsilon_0$ to be known and non zero. Note that the Gaussian assumption is only used to establish the consistency result (Theorem \ref{NPHMM:th:cons:MPLE}) which relies on an entropy control written for this particular choice of density function.  A few authors have studied the deconvolution problem with unknown noise distribution. In \cite{Comte:2011}, the estimation of the density of $X$ in the model $Y=X+\epsilon$ is performed without knowing the distribution $\epsilon$ and under mild assumptions on the smoothness of the underlying densities. However, \cite{Comte:2011} only considered real valued random variables and the estimation based on Fourier transform and bandwidth selection is hardly relevant in our model. The main difference between the model studied in this paper and classical convolution models is that the random vector $f_\star(X_k)$ does not necessarily have a density with respect to the Lebesgue measure on $\mathbb{R}^\ell$. As discussed in Section \ref{NPHMM:sec:sobo} (Corollary \ref{NPHMM:cons:image}),  under some  assumptions on $f_\star$, if the state-space $\mathbb{X}$ is a subset of $\mathbb{R}^m$ with $m<\ell$,  $f_\star(X_k)$ lies in a sub-manifold of dimension $m$ in $\mathbb{R}^\ell$ which has a null Lebesgue measure and then classical deconvolution tools do not apply here.

Let $b$ be a positive integer. For any sequence $(x_k)_{k\ge 0}$, define $\mathbf{x}_k\eqdef (x_{kb},\ldots, x_{(k+1)b-1})$ and for any function $f:\mathbb{X}\to\mathbb{R}^\dimY$, define $\mathbf{f}:\mathbb{X}^b\to\mathbb{R}^{b\dimY}$ by
\[
\mathbf{x}  =(x_0,\ldots,x_{b-1})\mapsto \mathbf{f}(\mathbf{x})\eqdef (f(x_{0}),\ldots, f(x_{b-1}))\eqsp.
\]
The distribution of $\mathbf{X}_{0}$ is assumed to have a density $\dens_{b,\star}$ with respect to the measure $\mu^{\otimes b}$ on $\mathbb{X}^b$ which lies in a set of probability densities  $\mathcal{D}_b$. For all $f\in\mathcal{F}$ and $\dens\in \mathcal{D}_b$, let $p_{f,\dens}$ be defined, for all $\mathbf{y}\in\mathbb{R}^{b\dimY}$, by
\begin{equation}
\label{NPHMM:eq:pf}
p_{f,\dens}(\mathbf{y}) \eqdef  \int    \dens(\mathbf{x}) \prod_{k=0}^{b-1}\varphi( y_k - f(x_k)) \mu^{\otimes b}\left(\mathrm{d}\mathbf{x}\right)\eqsp.
\end{equation}
Note that  $p_{\paramstar,\dens_{b,\star}}$ is the probability density of $\mathbf{Y}_{0}$ defined in \eqref{NPHMM:eq:model}: for all $\mathbf{y}\in\mathbb{R}^{b\dimY}$,
\begin{equation}
\label{NPHMM:eq:pstar}
p_\star(\mathbf{y}) \eqdef p_{\paramstar,\dens_{b,\star}}(\mathbf{y}) = \int    \dens_{b,\star}(\mathbf{x}) \prod_{k=0}^{b-1}\varphi( y_k - f_{\star}(x_k)) \mu^{\otimes b}\left(\mathrm{d}\mathbf{x}\right)\eqsp.
\end{equation}
The function $y_{0:nb -1} \mapsto \sum_{k=0}^{n-1} \ln p_{f,\dens}\left(\mathbf{y}_{k} \right)$
is referred to as the pseudo log-likelihood of the observations up to time  $nb-1$.  This paper introduces an estimation procedure based on the method of M-estimation presented in \cite{vandervaart:wellner:1996} and \cite{vandegeer:2000}. Consider a function $I : \mathcal{F}\rightarrow \mathbb{R}^+$ which characterizes  the complexity of functions in $\mathcal{F}$ and let $\delta_n$ and $\lambda_n$ be some positive numbers. Define the following $\delta_n$-Maximum Pseudo-Likelihood Estimator ($\delta_n$-MPLE) of $(\paramstar ,\dens_{b,\star})$: 
\begin{equation}\label{NPHMM:eq:fn} 
\left(\param{n}{} , \widehat{\dens}_{n} \right) \eqdef \underset{f\in\mathcal{F},\ \dens\in \mathcal{D}_b}{\mbox{argmax}^{\delta_n}} \left\{ \sum_{k=0}^{n-1} \ln p_{f,\dens}\left(\mathbf{Y}_{k} \right) -\lambda_{n} I(f)\right\}\eqsp,
\end{equation}
where $\underset{f\in\mathcal{F},\ \dens\in \mathcal{D}_b}{\mbox{argmax}^{\delta_n}}$ is one of the pairs $(f',\dens')$ such that 
\begin{equation*}
 \sum_{k=0}^{n-1} \ln p_{f',\dens'}\left(\mathbf{Y}_{k} \right) -\lambda_{n} I(f')
\ge \sup_{f\in\mathcal{F},\ \dens\in \mathcal{D}_b}\left\{ \sum_{k=0}^{n-1} \ln p_{f,\dens}\left(\mathbf{Y}_{k} \right) -\lambda_{n} I(f)\right\}-\delta_n\eqsp.
\end{equation*}
The consistency of the estimators is established using a control for empirical processes associated with mixing sequences. The $\Phi$-mixing coefficient between two $\sigma$-fields $\mathcal{U},\mathcal{V}\subset \mathcal{E}$ is defined in \cite{dedecker:2009} by
\[
\Phi(\mathcal{U},\mathcal{V}) \eqdef \sup_{\substack{U\in\mathcal{U},V\in\mathcal{V},\\ \mathbb{P}(U)>0}} \left|\frac{\mathbb{P}\left(U \cap V \right)}{\mathbb{P}(U)} -\mathbb{P}(V)\right|\eqsp.
\]  
The stationary process $(X_k)_{k\ge0}$ can be extended to a two-sided process $(X_k)_{k\in\mathbb{Z}}$ which is said to be $\Phi$-mixing when $\lim_{i\to\infty}\Phi^X_i = 0$ where, for all $i\ge 1$,
\begin{equation}
\label{eq:phix}
\Phi^X_i \eqdef \Phi\left(\sigma\left(X_k \ ; \ k\le 0 \right),\sigma\left(X_k \ ; \ k\ge i \right) \right)\eqsp, 
\end{equation}
$\sigma\left(X_k \ ; \ k\in C \right)$ being the $\sigma$-field generated by $(X_k)_{k\in C}$  for any $C\subset\mathbb{Z}$. As in \cite{samson:2000}, the required concentration inequality for the empirical process is established under the following assumption on the $\Phi$-mixing coefficients of $(X_k)_{k\ge0}$.

\begin{hypH}
\label{assum:phimix}
The stationary process $(X_k)_{k\ge0}$ satisfies $\mathbf{\Phi} \eqdef \sum_{i=1}^{\infty} (\Phi^X_i)^{1/2} < \infty$ where $\Phi^X_i$ is given by \eqref{eq:phix}.
\end{hypH}

\begin{remark}  
\begin{enumerate}[-]
\item If $(X_k)_{k\ge0}$ is i.i.d., then $\Phi^X_i = 0 $ for all $i\ge 1$ and H\ref{assum:phimix} is  satisfied.
\item Assume $(X_k)_{k\ge 0}$ is a stationary Markov chain with transition kernel $Q$ and stationary distribution $\pi$ such that there exist $\epsilon>0$ and a probability measure $\vartheta$ on $\mathbb{X}$ satisfying, for all $x\in\mathbb{X}$ and all $A\in\mathcal{X}$,
\[
Q(x,A)\ge\epsilon\vartheta(A)\eqsp.
\]
Then, by \cite[Theorem~$16.2.4$]{meyn:tweedie:1993}, for all $x\in\mathbb{X}$ and all $A\in\mathcal{X}$,
\[
\left|Q^n(x,A)-\pi(A)\right|\le(1-\varepsilon)^n\eqsp.
\]
Therefore, for all $n,k>0$ and $A,B\in\mathcal{X}$ such that $\pi(A)>0$,
\begin{align*}
\left|\mathbb{P}\left(X_{k+n}\in B\middle|X_k\in A\right)-\mathbb{P}\left(X_{k+n}\in B\right)\right| &= \left|\mathbb{P}\left(X_{k+n}\in B\middle|X_k\in A\right)-\pi\left(B\right)\right|\eqsp,\\
&\le\frac{1}{\pi(A)}\left|\int_{A}\left(Q^n(x,B)-\pi(B)\right)\pi(\mathrm{d}x)\right|\eqsp,\\
&\le (1-\varepsilon)^n\eqsp.
\end{align*}
The $\Phi$-mixing coefficients associated with $(X_k)_{k\ge 0}$ decrease geometrically and H\ref{assum:phimix} is satisfied.
\end{enumerate}
\end{remark}
 
\section{ General convergence results} 
\label{NPHMM:sec:consistency}
Denote by $\widehat{p}_{n}$  the estimator of $p_\star$ (defined in \eqref{NPHMM:eq:pstar}), given by 
\begin{equation}
\label{NPHMM:eq:MLE} 
\widehat{p}_n \eqdef p_{\param{n}{},\hat{\dens}_n}\eqsp,
\end{equation}
where $(\param{n}{},\hat{\dens}_n)$ is defined in \eqref{NPHMM:eq:fn}. The first step to prove the consistency of the estimators is to establish the convergence of $\widehat{p}_n$ to $p_{\star}$. The only assumption required on the penalization procedure is that the function $I$ is lower bounded by a power of the supremum norm.  
\begin{hypH}
\label{assum:pen}
There exist $C>0$ and $\upsilon>0$ such that for all $f\in\mathcal{F}$,
\begin{equation}
\label{NPHMM:hyp:I}
\|f\|_\infty \le CI(f)^\upsilon\eqsp,
\end{equation}
with, for any $f\in\mathcal{F}$, $\|f\|_{\infty}\eqdef\underset{1\leq j\leq\dimY}{\max}\;\;\underset{x\in\mathbb{X}}{\esssup} |f_j(x)|$.
\end{hypH}
Here, $\esssup$ denotes the essential supremum with respect to the measure $\mu$ on $\mathbb{X}$.  Note that if H\ref{assum:pen} holds, since $I:\mathcal{F}\to\mathbb{R}^+$,  for all $f\in\mathcal{F}$, $\|f\|_\infty \le CI(f)^\upsilon<\infty$. This is the only restrictive assumption on the penalty $I(f)$ which may be chosen arbitrarily as long as H\ref{assum:pen} holds.
\begin{hypH}
\label{assum:Db}
There exist $0<\nu_-<\nu_+<\infty$ such that, for all $\nu\in\mathcal{D}_b$, $\nu_-\le\nu\le \nu_+$.
\end{hypH}
The convergence  of $\widehat{p}_n$ to $p_{\star}$ is established using the Hellinger metric defined, for any probability densities $p_{1}$ and $p_{2}$ on $\Rset^{b\dimY}$, by
\begin{equation}
\label{eq:hellinger}
h(p_{1},p_{2}) \eqdef \left[\frac{1}{2}\int \left(p_{1}^{1/2}(y)-p_{2}^{1/2}(y)\right)^{2}\rmd y\right]^{1/2}\eqsp.
\end{equation}
Theorem~\ref{NPHMM:th:cons:MPLE} provides a rate of convergence of $\widehat{p}_n$ to $p_{\star}$ and a bound for the complexity $I(\param{n}{})$ of the estimator $\param{n}{}$.  
\begin{theorem}
\label{NPHMM:th:cons:MPLE}
Assume H\ref{assum:phimix}-\ref{assum:Db}  hold for some $\upsilon$ such that $b\dimY\upsilon<1$. Assume also that $\lambda_n$ and $\delta_n$ satisfy
\begin{equation}\label{NPHMM:eq:lambda}
 \lambda_{n} n^{-1} \underset{n\to +\infty}{\longrightarrow} 0 , \ \lambda_{n} n^{-1/2} \underset{n\to +\infty}{\longrightarrow} +\infty\text{ and } \delta_n  = O \left(\frac{\lambda_n}{n}\right)\eqsp.
\end{equation}
Then,
\begin{equation}\label{NPHMM:eq:hell}
h^{2}(\widehat{p}_{n},p_{\star}) =O_{\PP }\left( \frac{\lambda_{n}}{n}\right)\quad\mbox{and}\quad I(\param{n}{})= O_{\PP }(1)\eqsp.
\end{equation}
\end{theorem}
Condition \eqref{NPHMM:eq:lambda} implies that the rate of convergence of the Hellinger distance between $\widehat{p}_n$ and the true density $p_{\star}$ is slower than $n^{-1/4}$. The proof of the consistency of $\widehat{p}_n$ relies on the control of the empirical process:
\[
\sup_{f,\nu} \int \frac{1}{2}\ln \left[(p_{f,\nu} + p_\star)/(2p_\star)\right] \mathrm{d}\left(\PP_{n} - \Pr_{\star } \right)\eqsp,
\]
where $\PPim$ is the law of $\bfY_0$ and $\PP_n$ is the empirical distribution of the observations $\{\bfY_k\}_{k=0}^{n-1}$, given for any measurable set $A$ of $\mathbb{R}^{b\dimY}$ by
\[
\PP_{n}(A) \eqdef \frac{1}{n} \sum_{k=0}^{n-1} \1_{A}(\mathbf{Y}_{k})\eqsp.
\] 
A weaker condition on $\lambda_n$ could be obtained with a sharper deviation inequality on the empirical process. For instance, \cite[Theorem 10.6]{vandegeer:2000} estimates the density of a random variable $Y$ using i.i.d. samples and the penalized loglikelihood $p\mapsto \int \log p\ \mathrm{d}\PP_{n}  - \lambda_n I(p)$, where $I(p)= \int_{\mathbb{R}} (p^{(m)}(y))^2\mathrm{d}y$ penalizes the $m$-th derivative of $p$. The proof of \cite[equation (10.34)]{vandegeer:2000} establishes that 
\[
\sup_{p\in A_n(p_{\star})} \frac{\int \ln \left[(p + p_\star)/(2p_\star)\right]\mathrm{d}\left(\PP_{n} - \Pr_{\star } \right) }{1+I(p)+I(p_\star)} =O_{\mathbb{P}}(n^{-2m/(2m+1)})\eqsp,
\]
 where
\[
A_n(p_{\star}) \eqdef \left\{p\;;\;h(p,p_\star)\le n^{-m/(2m+1)}\left[1+I(p)+I(p_\star)\right]\right\}
\]
to obtain $n^{-m/(2m+1)}$ as rate of convergence for $h(\widehat{p}_n,p_\star)$. \cite{Gassiat:2013} also use a localization technique to derive the minimal penalty which ensures the convergence of the estimate of the number of components in a general mixture model. 
In our case, Proposition~\ref{NPHMM:prop:deviation:G} establishes a deviation result on the empirical process on the whole class of functions  $\left\{p_{f,\nu};\ f\in\mathcal{F}, \ \nu \in \mathcal{D}_b \right\}$. We consider a general setting where $\mathcal{F}$, $\mathcal{D}_b$ and the complexity function $I(f)$ are all non specified. Theorem~\ref{NPHMM:th:cons:MPLE} is established under the relatively mild assumptions H\ref{assum:phimix}-\ref{assum:Db}. Hence, the rate $n^{-1/4}$ corresponds to the "worst case" rate.  However, even in a less general context such as in Section~\ref{NPHMM:sec:sobo}, controlling a localized version of the empirical process in order to improve the rate of convergence of $\widehat{p}_n $ remains a difficult problem.

The proof of Theorem~\ref{NPHMM:th:cons:MPLE} relies on a \textit{basic inequality} which provides a simultaneous control of the Hellinger risk $h^{2}(\widehat{p}_{n},p_{\star})$ and of  $I(\param{n}{})$. Define for any density function $p$ on $\mathbb{R}^{b\dimY}$, 
\begin{equation}
\label{eq:defgp}
g_{p} \eqdef \frac{1}{2} \ln \frac{p + p_{\star}}{2p_{\star}}\eqsp.
\end{equation}
By \eqref{NPHMM:eq:fn} and \eqref{NPHMM:eq:MLE}, following  the proof of \cite[Lemma 10.5]{vandegeer:2000}:
\begin{equation}\label{NPHMM:eq:basic}
h^{2}(\widehat{p}_{n},p_{\star}) +4\lambda_{n} n^{-1} I (\param{n}{}) \le 16 \int g_{\widehat{p}_{n}} \rmd(\PP_{n} - \Pr_{\star }) + 4\lambda_{n} n^{-1} I(\paramstar)+\delta_n\eqsp.
\end{equation}
Therefore, a control of $\int g_{\widehat{p}_{n}} \rmd(\PP_{n} - \PPim)$ in the right hand side of \eqref{NPHMM:eq:basic} provides upper bounds for both  $h^{2}(\widehat{p}_{n},p_{\star}) $ and $I(\param{n}{})$. This control is given in Proposition~\ref{NPHMM:prop:deviation:G}.

\begin{proposition}
\label{NPHMM:prop:deviation:G}
Assume H\ref{assum:phimix}-\ref{assum:Db} hold. There exists a positive constant  $c$ such that, for any $\eta>0$, there exist  $A$ and $N$  such that for any  $n\ge N$ and any $x>0$,
\begin{align*}
\mathbb{P}\bigg[ \sup_{f\in\mathcal{F}, \ \dens\in\mathcal{D}_b } \frac{\left|\int g_{p_{f,\dens}} \ \rmd(\PP_n - \PPim)\right| }{1\vee I(f)^{\gamma}}&\ge   c\mathbf{\Phi}\times\left(\sqrt{\frac{  x} {n}}+ \frac{ x }{n}\right)+ \frac{A}{\sqrt{n}} \bigg]\le  \frac{2e^{- \alpha x }}{1-e^{- \alpha x }}\eqsp, 
\end{align*}
where $\gamma \eqdef b\dimY\upsilon + \eta$ and $\alpha \eqdef 2^{-2\gamma}(\gamma-\upsilon)\log(2) = 2^{-2(b\dimY\upsilon + \eta)}\left[(b\dimY-1)\upsilon + \eta\right]\log(2)$.
\end{proposition}
Proposition~\ref{NPHMM:prop:deviation:G} is proven in Section.~\ref{sec:proof:NPHMM:prop:deviation:G}.
 
\begin{proof}[Proof of Theorem~\ref{NPHMM:th:cons:MPLE}]
Since $\upsilon^{-1}>b\dimY$, $\eta>0$ in Proposition~\ref{NPHMM:prop:deviation:G} can be chosen such that $\gamma= b\dimY\upsilon + \eta = 1$. For this choice of $\eta$,   
Proposition \ref{NPHMM:prop:deviation:G} implies that  
\[
\frac{\int g_{\widehat{p}_{n}} \rmd(\PP_{n} - \PPim)}{1\vee I(\param{n}{}) } = O_{\mathbb{P}}(n^{-1/2})\eqsp.
\] 
Combined with \eqref{NPHMM:eq:basic}, this yields
\begin{equation} 
\label{NPHMM:eq:basic:2}
h^{2}(\widehat{p}_{n},p_{\star}) +4\lambda_{n} n^{-1} I (\param{n}{}) \le(1\vee I(\param{n}{}) )O_{\mathbb{P}}(n^{-1/2})  + 4\lambda_{n} n^{-1} I(\paramstar)+\delta_n\eqsp.
\end{equation}
Then, \eqref{NPHMM:eq:basic:2} directly implies that 
\begin{equation*} 
 4\  I (\param{n}{}) \le(1\vee I(\param{n}{}) )O_{\mathbb{P}}(n^{1/2} \lambda_n^{-1} )  + 4  I(\paramstar)+\delta_n n\lambda_n^{-1}\eqsp,
\end{equation*}
which, together with \eqref{NPHMM:eq:lambda}, gives
\[ 
I (\param{n}{}) =O_{\mathbb{P}}(1)\eqsp.
\]
Combining this result with  \eqref{NPHMM:eq:basic:2} again leads to
\begin{equation*}  
h^{2}(\widehat{p}_{n},p_{\star}) +O_{\mathbb{P}}(\lambda_{n} n^{-1})  \le O_{\mathbb{P}}(n^{-1/2})  + 4\lambda_{n} n^{-1} I(\paramstar)+\delta_n\eqsp.
\end{equation*}
This concludes the proof of Theorem \ref{NPHMM:th:cons:MPLE}.
\end{proof}

Theorem~\ref{NPHMM:th:cons:MPLE} shows that $h^{2}(\widehat{p}_{n},p_{\star})$ vanishes as $n$ goes to infinity. However, this does not imply the convergence of $(\param{n}{},\widehat{\nu}_n)$ to $(\paramstar,\dens_{b,\star})$. The convergence of the estimators $(\param{n}{},\widehat{\nu}_n)$ is addressed in the case where the  set $\mathcal{D}_b$ may be written as 
\begin{equation}
\label{NPHMM:eq:defA}
\mathcal{D}_b = \left\{\nu_a ;\ a \in \mathcal{A} \right\}\eqsp,
\end{equation}
where $\mathcal{A}$ is a parameter set not necessarily of finite dimension. The $\delta_n$-MPLE is then given by:
\[
(\param{n}{},\widehat{a}_n) \eqdef \underset{f\in\mathcal{F},\ a \in \mathcal{A}}{\mbox{argmax}^{\delta_n}} \left\{\sum_{k=0}^{n-1} \ln p_{f,\dens_a}\left(\mathbf{Y}_{k} \right) -\lambda_{n} I(f)\right\}\eqsp.
\]
\begin{hypH}
\label{NPHMM:hyp:compacite}
\begin{enumerate}[a)]
\item $\mathcal{A}$ is endowed with a distance $d_{\mathcal{A}}$ such that $\mathcal{A}$ is compact with respect to the topology defined by $d_{\mathcal{A}}$,\label{NPHMM:assum:A}
\item $\mathcal{F}$ is endowed with a metric $d_{\mathcal{F}}$ such that $\mathcal{F}_M\eqdef \left\{f \in\mathcal{F}; \ I(f)\le M \right\}$ is compact for all $M>0$ with respect to the topology defined by $d_{\mathcal{F}}$,\label{NPHMM:assum:Fm}
\item The function $(f,a)\mapsto h^2(p_{f, \dens_a},p_\star)$ is continuous with respect to the topology on $\mathcal{F}\times\mathcal{A}$ induced by the product distance $d$ on $\mathcal{F}\times \mathcal{A}$.
\label{NPHMM:assum:h}
\end{enumerate}
\end{hypH} 
 Corollary~\ref{NPHMM:th:cons:esti} establishes the convergence of $(\param{n}{},\widehat{a}_n)$ to the set  $\mathcal{E}_\star$ defined as:

\begin{equation}
\label{NPHMM:eq:estar}
\mathcal{E}_\star\eqdef\left\{(f,a)\in \mathcal{F} \times \mathcal{A} ;\; p_{f,\nu_a} = p_{\paramstar,\nu_{a_{\star}}}\right\}\eqsp.
\end{equation}
Define for all $(f,a)\in\mathcal{F} \times \mathcal{A}$, $$d\left((f,a), \mathcal{E}_\star\right) = \inf_{(f',a') \in\mathcal{E}_\star} d\left( (f,a),(f',a')\right)\;.$$
\begin{corollary}
\label{NPHMM:th:cons:esti}
Assume H\ref{assum:phimix}-\ref{NPHMM:hyp:compacite} hold for some $\upsilon$ such that $\upsilon b\dimY<1$. Assume also that $\lambda_n$ and $\delta_n$ satisfy
\begin{equation*}
 \lambda_{n} n^{-1} \underset{n\to +\infty}{\longrightarrow} 0 , \ \lambda_{n} n^{-1/2} \underset{n\to +\infty}{\longrightarrow} +\infty\text{ and } \delta_n  = O\left(\frac{\lambda_n}{n}\right)\eqsp.
\end{equation*}
 Then,
\[
d\left((\param{n}{},\widehat{a}_n), \mathcal{E}_\star\right) =o_{\mathbb{P}}(1)\eqsp.
\]
\end{corollary}
Corollary~\ref{NPHMM:th:cons:esti} is a direct consequence of Theorem~\ref{NPHMM:th:cons:MPLE} and of the properties of $d_{\mathcal{A}}$ and $d_{\mathcal{F}}$ and its proof is therefore omitted. The few assumptions on the model allow only to establish the convergence of the estimators $(\param{n}{},\widehat{a}_n)$ to the set $\mathcal{E}_\star$ in Corollary~\ref{NPHMM:th:cons:esti}.

\section{Identifiability when $\mathbb{X}$ is a subset of $\mathbb{R}^m$}
\label{NPHMM:sec:identifiability}
The aim of this section is to characterize the set $\mathcal{E}_\star$ given by \eqref{NPHMM:eq:estar} when $b=1$ and when $b=2$ (the characterization of $\mathcal{E}_\star$ when $b>2$ follows the same lines) with some additional assumptions on the model, on $\mathcal{F}$ and on $\mathcal{D}_b$.  In the sequel, $\nu_{b,\star}$ must satisfy $0<\nu_-\le\nu_{b,\star}\le \nu_+$ for some constants $\nu_-$ and $\nu_+$. It is assumed that $\mathbb{X}$ is a subset of $\mathbb{R}^m$ for some $m\ge 1$ and that $\mu$ is  the Lebesgue measure. For any subset $A$ of $\mathbb{R}^m$, $\overset{\circ}{A}$ stands for the interior of $A$ and $\overline{A}$ for the closure of $A$. Consider the following assumptions on the state-space $\mathbb{X}$.
\begin{hypH}
\label{assum:X}
\begin{enumerate}[a)]
\item \label{assum:X:compact}$\mathbb{X}$ is  non empty, compact and $\overline{\overset{\circ}{\mathbb{X}}} = \mathbb{X}$,
\item \label{assum:X:topo} $\mathbb{X}$ is  arcwise and simply connected.
\end{enumerate}
\end{hypH}

The compactness  implies that $\mathbb{X}$ is closed and that continuous functions on $\mathbb{X}$ are bounded.  By the last assumption of H\ref{assum:X}\ref{assum:X:compact}), the interior of $\mathbb{X}$ is not empty and any element in $\mathbb{X}$ is the limit of elements of the interior of $\mathbb{X}$. Finally, $\mathbb{X}$ is  arcwise and simply connected to ensure topological properties used in the proofs of the identifiability  results below.

A function $f:U\to f(U)\subset \mathbb{R}^\ell$ defined on an open subset $U$ of $\mathbb{R}^m$ is a $\mathcal{C}^1$-diffeomorphism  if its differential function $x\mapsto D_x f$ is continuous and if, for all $x$ in $U$, $\mathrm{rank}(D_xf) = m$. 
A function $f:\mathbb{X} \to f(\mathbb{X})$ is said to be $\mathcal{C}^1$ (resp. a $\mathcal{C}^1$-diffeomorphism) if $f$ is the restriction to $\mathbb{X}$ of a $\mathcal{C}^1$ function (resp. a $\mathcal{C}^1$-diffeomorphism) defined on an open neighborhood of $\mathbb{X}$ in $\mathbb{R}^m$. 
\begin{hypH}
\label{assum:fstar}
$\paramstar$ is a $\mathcal{C}^1$-diffeomorphism from $\mathbb{X}$ to $\paramstar(\mathbb{X})$.
\end{hypH}

H\ref{assum:fstar} might be seen as a restrictive assumption. Nevertheless, when $\ell\ge 2m+1$, by Whitney's embedding theorem (\cite{whitney:1986}) every continuous function from $\mathbb{X}$ to $\mathbb{R}^\ell$ can be approximated by a smooth embedding. In the case $b=1$, Proposition~\ref{NPHMM:prop:nu:phi} discusses the identifiability when $\mathcal{F}$ is a subset of $\mathcal{C}^1$. For all differential function $\phi: \mathbb{X}\to \mathbb{X}$, let $J_{\phi}$ be the determinant of the Jacobian matrix of $\phi$: $J_{\phi}(x) \eqdef \det\left(D_{x} \phi \right)$.

\begin{proposition}[{\bf b=1}]
\label{NPHMM:prop:nu:phi}
Assume that H\ref{assum:Db}-\ref{assum:fstar} hold. Let $f\in\mathcal{C}^1$ and let $\nu\in \mathcal{D}_b$. Then, $p_{f,\dens} = p_{\paramstar,\dens_{1,\star}}$ if and only if $f_\star$ and $f$ have the same image in $\mathbb{R}^\dimY$, $\phi = \paramstar^{-1} \circ f $ is bijective and, for $\mu$ almost every $x\in\mathbb{X}$, $$\dens(x)= |J_{\phi}(x)|\dens_{1,\star}(\phi(x))\eqsp. $$
\end{proposition}
The proof of Proposition~\ref{NPHMM:prop:nu:phi} is given in Section~\ref{NPHMM:sec:proofs:prop:nu:phi}. 

\begin{remark}
\label{NPHMM:rem:b1}
Proposition~\ref{NPHMM:prop:nu:phi} states that $(f,\nu)$ is related to $(\paramstar,\nu_{1,\star})$ through the bijective state-space transformation $\phi$. In the particular case where $\mathbb{X}=[0,1]$ ($m=1$), Proposition~\ref{NPHMM:prop:nu:phi} implies a sharper result. Assume that $\mathcal{D}_1 = \{\nu_{1,\star}=1\}$ ($\nu_{1,\star}$ is the uniform distribution density and is known). Then, Proposition \ref{NPHMM:prop:nu:phi} implies the existence of a $\mathcal{C}^1$ and bijective function $\phi$ satisfying $f = \paramstar\circ\phi$ and $ |J_\phi| = 1$. Hence, $\phi: x \mapsto x$ or $\phi: x\mapsto 1-x$ which are the two isometric transformations of $[0,1]$. 

This cannot be extended to the case $m>1$ where $|J_\phi| = 1$ does not necessarily imply that $\phi$ is isometric but only that $\phi$ preserves volumes. 
\end{remark}
Proposition~\ref{NPHMM:prop:ident:Markov} establishes the identifiability of the model when $b=2$.  In this case, $\nu_{2,\star}$ can be written $\nu_{2,\star}(x,x') =\nu_\star(x)q_\star(x,x')$ where $q_{\star}$ is a transition density with (unique) stationary probability density $\dens_{\star}$. For any transition density $q$ on $\mathbb{X}^2$ satisfying 
\begin{equation}
\label{NPHMM:eq:cond:q}
\text{for all } x,x'\in\mathbb{X}\;,0<q_{-}\le   q(x,x') \le q_{+} \;,
\end{equation}
there exists a stationary density $\nu$ associated with $q$  satisfying, for all $x \in\mathbb{X}$, $q_{-}\le   \nu(x ) \le q_{+}$. Denote by $\dens_q$ this density. 

\begin{proposition}[{\bf b=2}]
\label{NPHMM:prop:ident:Markov}
Assume that H\ref{assum:X} and H\ref{assum:fstar} hold. Let $f\in\mathcal{C}^1$ and $q$ be a transition density satisfying \eqref{NPHMM:eq:cond:q}. Let $\dens_2(x,x') =\dens_q(x)q(x,x')$. Then, $p_{f,\dens_2} = p_{\paramstar,\dens_{2,\star}}$ if and only if $f_\star$ and $f$ have the same image in $\mathbb{R}^\dimY$, $\phi = \paramstar^{-1} \circ f $ is bijective and $\mu\otimes\mu$ almost everywhere in $\mathbb{X}^2$,  
\begin{equation}
\label{NPHMM:eq:Q:phi}
q(x,x')= |J_{\phi}(x')| q_\star(\phi(x),\phi(x'))\;.
\end{equation} 
\end{proposition}
Proposition~\ref{NPHMM:prop:ident:Markov} is proved in Section~\ref{NPHMM:sec:proof:ident:Markov}.
\begin{corollary}
\label{NPHMM:cor:markov}
Assume that the same assumptions as in Proposition~\ref{NPHMM:prop:ident:Markov} hold. Assume in addition that $q_\star$ and $q$ are of the form:
\begin{equation*}
q_\star(x,x') = c_\star(x)\rho_{\star}(||x-x'||)\eqsp,\quad
q (x,x') = c(x)\rho(||x-x'||)\eqsp,
\end{equation*}
where $\rho$ and $\rho_\star$ are two continuous functions  defined on $\mathbb{R}_+$. If in addition $\rho_\star$ is one-to-one then, $p_{f,\nu_2} = p_{\paramstar,\nu_{2,\star}}$ if and only if $f_\star$ and $f$ have the same image in $\mathbb{R}^\dimY$, $\phi = \paramstar^{-1} \circ f $  is an isometry on $\mathbb{X}$ and $q=q_\star$.
\end{corollary}

\section{Application when $\mathcal{F}$ is a Sobolev class of functions}
\label{NPHMM:sec:sobo}
In this section, $\mathbb{X}$ is a subset of $\mathbb{R}^m$, $m\ge 1$ and the results of Section~\ref{NPHMM:sec:consistency} and Section~\ref{NPHMM:sec:identifiability} are applied to a specific class of functions $\mathcal{F}$ with an example of complexity function $I$ satisfying H\ref{assum:pen} and the compacity assumption  H\ref{NPHMM:hyp:compacite}-\ref{NPHMM:assum:Fm}). Let $p\ge 1$, define
 \[
\rmL^{p}\eqdef\left\{f:\mathbb{X}\to\Rset^{\dimY} \ ; \ \norminf{f}{\rmL^{p}}^{p} =\int_{\mathbb{X}} \norm{f(x)}^{p} \mu(\rmd x) < \infty\right\}\eqsp.
\]
For any $f:\mathbb{X}\to \Rset^{\dimY}$ and any $j \in\{1,\cdots,\dimY\}$, the $j^{th}$ component of $f$ is denoted by $f_{j}$. For any vector $\alpha \eqdef \{\alpha_{i}\}_{i=1}^{\dimK}$ of non-negative integers, we write $|\alpha| \eqdef \sum_{i=1}^{\dimK}\alpha_{i}$ and $D^{\alpha}f : \mathbb{X} \to \Rset^{\dimY} $ for the vector of partial derivatives of order $\alpha$ of $f$  in the sense of distributions. Let $s\in\Nset$ and $W^{s,p}$ be the Sobolev space on $\mathbb{X}$ with parameters $s$ and $p$, {\em i.e.},
\begin{equation}
\label{NPHMM:eq:sobolev}
W^{s,p} \eqdef \left\{f\in \rmL^{p};\; D^{\alpha}f\in \mathrm{L}^{p}, \alpha\in\Nset^{\dimK}\; \mbox{and}\;|\alpha|\leq s\right\}\eqsp.
\end{equation}
$W^{s,p}$ is endowed with the norm $\norminf{\cdot}{W^{s,p}}$ defined, for any $f \in W^{s,p}$, by
\begin{equation}
\label{eq:sobo:norm}
\norminf{f}{W^{s,p}} \eqdef\left(\sum\limits_{0\le \vert \alpha\vert \le s} \norminf{D^{\alpha}f}{\rmL^{p}}^{p}\right)^{1/p}\eqsp.
\end{equation}
For any $j \in\{1,\cdots,\dimY\}$ and $f\in W^{s,p}$, $f_{j}$ belongs to $W^{s,p}(\mathbb{X},\Rset)$, the Sobolev space of real-valued functions with parameters $s$ and $p$. For all $k,q\ge 0$, define $\mathcal{C}^{k}(\mathbb{X},\mathbb{R}^q)$, the set of functions $f : \mathbb{X} \to \mathbb{R}^{q}$ which, together with all their partial derivatives $D^{\alpha}f$ of orders $|\alpha|\le k$  are continuous on $\mathbb{X}$. For any $f\in \mathcal{C}^{k}(\mathbb{X},\mathbb{R}^q)$ define
\[
\norminf{f}{\mathcal{C}^{k}(\mathbb{X},\mathbb{R}^q)} \eqdef \max_{0\le |\alpha| \le k} \ \sup_{x\in \mathbb{X}} \left|D^{\alpha}f(x)\right|\eqsp.
\]
In the particular case $q = \dimY$, write $\mathcal{C}^{k} \eqdef \mathcal{C}^{k}(\mathbb{X},\mathbb{R}^{\dimY})$.
The results of Section~\ref{NPHMM:sec:consistency} and Section~\ref{NPHMM:sec:identifiability} can be applied to the class $\mathcal{F} = W^{s,p}$ under the following assumption.
\begin{hypH}
\label{assum:X:Lip}
$\mathbb{X}$ has a locally Lipschitz boundary.
\end{hypH}
H\ref{assum:X:Lip} means that all $x$ on the boundary of $\mathbb{X}$ has a neighbourhood whose intersection with the boundary of $\mathbb{X}$ is the graph of a Lipschitz function.  

Let $k\ge 0$, by \cite[Theorem 6.3]{adams:fournier:2003}, if $s > \dimK /p +k$ and if H\ref{assum:X}-\ref{assum:X:compact}) and H\ref{assum:X:Lip} hold,  $W^{s,p}(\mathbb{X},\Rset)$ is compactly embedded into $\left(\mathcal{C}^{k}(\mathbb{X},\mathbb{R}), \norminf{\cdot}{\mathcal{C}^{k}(\mathbb{X},\mathbb{R})}\right)$. Arguing component by component, $W^{s,p}$ is compactly embedded into $\mathcal{C}^{k}$. Moreover, the identity function $id: W^{s,p} \to \mathcal{C}^{k}$ being linear and continuous, there exists a positive coefficient $\kappa$ such that, for any $f \in W^{s,p}$,
\begin{equation}
\label{eq:def:kappa}
\norminf{f}{\mathcal{C}^{k}} \le \kappa \norminf{f}{W^{s,p}}\eqsp.
\end{equation}  
Then, if $s>m/p+k$, for any $f \in\mathcal{F} =  W^{s,p}$,
\begin{equation}
\norminf{f}{\infty} \le \kappa \norminf{f}{W^{s,p}}\eqsp.
\end{equation}  
In the following, $d_{\mathcal{C}^k}$ is the usual distance on $\mathcal{C}^k$ associated with $\norminf{\cdot}{\mathcal{C}^{k}} $. If $\mathcal{F} = W^{s,p}$ and if the complexity function is defined by $I(f) =  \norminf{f}{W^{s,p}}^{1/\upsilon}$ with $\upsilon b\dimY<1$, then H\ref{assum:pen} holds and Theorem~\ref{NPHMM:th:cons:MPLE} can be applied. Moreover, by \cite[Theorem 6.3]{adams:fournier:2003}, the subspace $\mathcal{F}_M$, $M\ge1$ are quasi-compact in $\mathcal{C}^k$ and H\ref{NPHMM:hyp:compacite}-\ref{NPHMM:assum:Fm}) holds. Let $d_{\mathcal{A}}$ be a metric on the space $\mathcal{A}$ introduced in \eqref{NPHMM:eq:defA} such that H\ref{NPHMM:hyp:compacite}-\ref{NPHMM:assum:A}) holds and that, for $\mu\otimes\mu$ almost every $(x,x')\in\mathbb{X}^2$, $a\mapsto \dens_a(x,x')$ is continuous. By the dominated convergence theorem, this implies that H\ref{NPHMM:hyp:compacite}-\ref{NPHMM:assum:h}) holds. Define
\[\mathcal{F}_\star \eqdef  \left\{f \in W^{s,p};\ \mbox{there exists } a \in \mathcal{A}\mbox{ such that } (f,a)\in \mathcal{E}_\star \right\}\;.\]
Then, Proposition~\ref{NPHMM:consit:sobo1} is a direct application of Corollary~\ref{NPHMM:th:cons:esti}.
\begin{proposition}[{\bf $\mathbf{\mathcal{F} = W^{s,p}, \ s>m/p+k, \ k\ge0}$}]
\label{NPHMM:consit:sobo1} 
 Assume that H\ref{assum:phimix}, H\ref{assum:Db}, H\ref{assum:X}\ref{assum:X:compact}) and  H\ref{assum:X:Lip} hold. Assume also that $I(f) =  \norminf{f}{W^{s,p}}^{1/\upsilon}$ for some $\upsilon$ such that $\upsilon b\dimY<1$ and that $\lambda_n$ and $\delta_n$ satisfy
\begin{equation*}
 \lambda_{n} n^{-1} \underset{n\to +\infty}{\longrightarrow} 0 , \ \lambda_{n} n^{-1/2} \underset{n\to +\infty}{\longrightarrow} +\infty\text{ and } \delta_n  = O\left(\frac{\lambda_n}{n}\right)\eqsp.
\end{equation*}
Then,
\[
d_{\mathcal{C}^k}\left( \param{n}{} , \mathcal{F}_\star\right) =o_{\mathbb{P}}(1)\eqsp.
\]
\end{proposition}
Moreover, as shown in Section \ref{NPHMM:sec:proofs:prop:nu:phi}, the assumption $\overline{\overset{\circ}{\mathbb{X}}} = \mathbb{X}$ together with the continuity of the functions in $\mathcal{F}$ provided by \eqref{eq:def:kappa} imply that for any $f$ in $\mathcal{F}_\star$, $f(\mathbb{X}) = \paramstar(\mathbb{X})$ (see the proof in Section \ref{NPHMM:sec:proofs:prop:nu:phi}). 
Define the Hausdorff distance $d_\mathcal{H}(A,B)$ between two compact subsets $A$ and $B$ of $\mathbb{R}^\ell$ as 
\[d_\mathcal{H}(A,B) \eqdef \max \left ( \sup_{a\in A}  \inf_{b\in B} ||a-b||_{\mathbb{R}^\ell} \ ,\ \sup_{b\in B}  \inf_{a\in A}||a-b||_{\mathbb{R}^\ell} \right)\eqsp.
\]
Proposition \ref{NPHMM:consit:sobo1} implies Corollary \ref{NPHMM:cons:image}.
\begin{corollary}[{\bf $\mathbf{\mathcal{F} = W^{s,p}, \ s>m/p} $}]
\label{NPHMM:cons:image}
Assume that H\ref{assum:phimix}, H\ref{assum:Db}, H\ref{assum:X}\ref{assum:X:compact}) and  H\ref{assum:X:Lip} hold. Assume also that $I(f) =  \norminf{f}{W^{s,p}}^{1/\upsilon}$ for some $\upsilon$ such that $\upsilon b\dimY<1$ and that $\lambda_n$ and $\delta_n$ satisfy
\begin{equation*}
 \lambda_{n} n^{-1} \underset{n\to +\infty}{\longrightarrow} 0 , \ \lambda_{n} n^{-1/2} \underset{n\to +\infty}{\longrightarrow} +\infty\text{ and } \delta_n  = O\left(\frac{\lambda_n}{n}\right)\eqsp.
\end{equation*}
Then,
\[
d_{\mathcal{\mathcal{H}}}\left( \param{n}{}(\mathbb{X}) , \paramstar(\mathbb{X})\right) =o_{\mathbb{P}}(1)\eqsp.
\]
\end{corollary}
Corollary \ref{NPHMM:cons:image} establishes the consistency of the estimator  $\param{n}{}(\mathbb{X})$ of the image of $f_\star$ in $\mathbb{R}^\ell$. This result is particularly interesting since  $\paramstar(\mathbb{X})$ is a manifold of dimension smaller than $\ell$ in $\mathbb{R}^\ell$. The proposed estimation procedure allows to approximate such manifolds of possibly low dimensions and only observed with additive noise in $\mathbb{R}^\ell$. Moreover, this result holds under relatively weak assumptions on the manifold. Since the identifiability of $f_\star$ is not necessary to have the identifiability of $f_\star(\mathbb{X})$, $\paramstar$ is not assumed to be bijective to establish this result. 

Proposition~\ref{NPHMM:th:cons:sobo} below states the consistency of the estimators $(\widehat{f}_n,\widehat{a}_n)$ in the case $b=2$ and $\mathcal{F} = W^{s,p}$. Assume that for any $a$ in $\mathcal{A}$,  $\nu_a \in \mathcal{D}_2 $ is of the form 
\[
\nu_a(x,x') =\nu_{q_a}(x) q_a(x,x') \mbox{ with }q_a(x,x') = c_a(x)\rho_a(||x-x'||)\eqsp,
\]
where $\rho_-\le\rho_a \le \rho_+$. It is also assumed that there exists a unique $a_{\star}\in\mathcal{A}$ such that $\nu_{2,\star} = \nu_{a_{\star}}$ and that $\rho_{a_{\star}}$ is one-to-one.
Proposition~\ref{NPHMM:th:cons:sobo} is a direct application of Corollary~\ref{NPHMM:th:cons:esti} and of Proposition \ref{NPHMM:prop:ident:Markov} and is stated without proof.

\begin{proposition}[{\bf $\mathbf{\mathcal{F} = W^{s,p}, \ s>m/p+k, \ k\ge 1, \ b=2}$}]
\label{NPHMM:th:cons:sobo}
Assume that H\ref{assum:phimix}, H\ref{assum:Db} and  H\ref{assum:X} - \ref{assum:X:Lip} hold. Assume also that $I(f) =  \norminf{f}{W^{s,p}}^{1/\upsilon}$ for some $\upsilon$ such that $2\upsilon\dimY<1$ and that $\lambda_n$ and $\delta_n$ satisfy
\begin{equation*}
 \lambda_{n} n^{-1} \underset{n\to +\infty}{\longrightarrow} 0 , \ \lambda_{n} n^{-1/2} \underset{n\to +\infty}{\longrightarrow} +\infty\text{ and } \delta_n  = O\left(\frac{\lambda_n}{n}\right)\eqsp.
\end{equation*}
Then,
\[
\mathcal{F}_{\star}=\left\{\paramstar\circ\phi;\; \phi\;\mbox{ is an isometry of } \mathbb{X}\right\}\eqsp,
\]
and
\[
d_{\mathcal{C}^k}\left(\param{n}{},\mathcal{F}_{\star}\right) =o_{\mathbb{P}}(1)\quad\mbox{and}\quad d_{\mathcal{A}}\left(\widehat{a}_n,a_{\star}\right) =o_{\mathbb{P}}(1)\eqsp,
\]
\end{proposition}

\section{Numerical experiments}
\label{NPHMM:sec:sobo:simu}
This section provides a practical implementation of the estimation procedure proposed in Section~\ref{NPHMM:sec:model}. The algorithm is applied in the cases $b=1$ and $b=2$ to assess the consistency and identifiability results with simulated data. When $b=2$, the hidden chain is assumed to be a Markov chain with a parametric transition kernel of the form $q_\star(x,x') = C_{a_{\star}}(x)\mathrm{exp}(-\|x'-x\|/a_{\star})$. This particular case is motivated by the recent work of \cite{dumont:lecorff:2014} where the same assumption on the hidden chain is made to perform indoor simultaneous localization and mapping based on WiFi signals. The process $(X_k)_{k \ge 0}$ is the position of a mobile evolving in a building and receiving the signal strengths $Y_k$ which satisfy \eqref{NPHMM:eq:model} at each time step $k$.

In Section~\ref{NPHMM:sec:EMalgo}, a generic EM based procedure is introduced to solve the inference problem detailed in Section~\ref{NPHMM:sec:model}. In Section~\ref{sec:exp:results}, we intend to apply this algorithm in the Sobolev setting of Section~\ref{NPHMM:sec:sobo} with a penalization function $I(f)$ based on the Sobolev norm $\|\cdot\|_{W^{2,2}}$. The assumptions required to obtain the identifiability and consistency results lead to a penalization term of the form $I(f) =\|f\|_{W^{2,2}}^{\vartheta}$ where $\vartheta>2b$. As explained in Section~\ref{sec:exp:results}, the M-step of the EM algorithm is intractable in this case while it can be efficiently performed under weaker assumptions (e.g. when $I(f)$ is based on the $\mathrm{L}^2$ norm of $f''$). Therefore, the proposed procedure weakens this assumption to illustrate the identifiability and consistency results. In particular, the convergence observed in the simulations of Section~\ref{sec:exp:results} seems to indicate that assumption H\ref{assum:pen} could be weakened. 
 
\subsection{Proposed Expectation Maximization algorithm}
\label{NPHMM:sec:EMalgo}
This section introduces a practical algorithm to compute the estimators defined in \eqref{NPHMM:eq:fn} when $\delta_n$ is set to zero. It is assumed that the maximizer in \eqref{NPHMM:eq:fn} exists which is the case for instance in the Sobolev framework of Section \ref{NPHMM:sec:sobo} and if $\mathcal{D}_b$ is compact. This proposed Expectation-Maximization (EM) based procedure iteratively produces a sequence of estimates $\widehat{\nu}^t$, $\param{}{t}$, $t\ge 0$, see \cite{dempster:laird:rubin:1977}. Assume that the current parameter estimates are given by $\widehat{\nu}^t$ and  $\param{}{t}$. The estimates $\widehat{\nu}^{t+1}$ and $\param{}{t+1}$ are defined as one of the maximizers of the function $Q$:
\begin{equation*}
(\nu,f)\mapsto Q((\nu,f),(\widehat{\nu}^t,\param{}{t})) \eqdef \sum_{k=0}^{n-1} \Esp_{\widehat{\nu}^t,\param{}{t}}\left[\ln p_{f,\nu}\left(\mathbf{X}_{k},\mathbf{Y}_{k}\right) \big\vert\mathbf{Y}_{k}\right] -\lambda_{n}I(f)\eqsp,
\end{equation*}
where $\Esp_{\widehat{\nu}^t,\param{}{t}}\left[\cdot\right]$ denotes the conditional expectation under the  model parameterized by $\widehat{\nu}^t$ and $\param{}{t}$ and where, for any $\mathbf{x}=(x_0,\ldots,x_{b-1})\in\mathbb{X}^{b}$ and any $\mathbf{y}=(y_0,\ldots,y_{b-1})\in\mathbb{R}^{\ell b}$,
\begin{equation*}
 p_{f,\nu}\left(\mathbf{x},\mathbf{y}\right) \eqdef \nu(\mathbf{x})\prod_{i=0}^{b-1}\denseps(y_i - f(x_i)) \eqsp.
\end{equation*}
Note that the intermediate quantity $Q((\nu,f),(\widehat{\nu}^t,\param{}{t}))$ can be written:
\[
Q((\nu,f),(\widehat{\nu}^t,\param{}{t})) = Q_t^1(\nu) + Q_t^2(f)\eqsp,
\]
where
\begin{align}
Q_t^1(\nu) &\eqdef \sum_{k=0}^{n-1} \Esp_{\widehat{\nu}^t,\param{}{t}}\left[\ln\left\{\nu(\mathbf{X}_k)\right\} \big\vert \mathbf{Y}_{k}\right]\eqsp, \label{NPHMM:eq:Q1}\\
Q_t^2(f) &\eqdef\sum_{k=0}^{n-1} \Esp_{\widehat{\nu}^t,\param{}{t}}\left[\ln \left\{\prod_{i=0}^{b-1}\denseps\left(Y_{bk+i} - f(X_{bk+i})\right)\right\}  \Bigg\vert \mathbf{Y}_{k}\right]-\lambda_{n} I(f)\eqsp.\label{NPHMM:eq:Q2}
\end{align}
Therefore $\widehat{\nu}^{t+1}$ is obtained by maximizing the function $\nu\mapsto Q_t^1(\nu)$  and $\param{}{t+1}$ by maximizing the function $f\mapsto Q_t^2(f)$.  
Lemma~\ref{NPHMM:lem:EM} proves that the penalized pseudo-likelihood increases at each iteration of this EM based algorithm. Its proof is postponed to Appendix~\ref{se:proof:EM}.
\begin{lemma}     
\label{NPHMM:lem:EM}
The sequences $\widehat{\nu}^t$ and $\param{}{t}$ satisfy
\[
\sum_{k=0}^{n-1} \ln p_{\param{}{t+1},\widehat{\nu}^{t+1}}\left(\mathbf{Y}_{k} \right) -\lambda_{n} I(\param{}{t+1})\ge \sum_{k=0}^{n-1} \ln p_{\param{}{t},\widehat{\nu}^{t}}\left(\mathbf{Y}_{k} \right) -\lambda_{n} I(\param{}{t})\eqsp.
\]
\end{lemma}
\begin{rem}
Like for all EM or gradient based procedures, there is no guarantee that the sequence $(\widehat{f}^t,\widehat{\nu}^t)_{t\ge 0}$ converges, when $t$ grows to infinity, towards the target estimate:
\[
(\param{n}{},\widehat{\nu}_n)=\argmax_{f,\nu}\left\{\sum_{k=0}^{n-1} \ln p_{f,\nu}\left(\mathbf{Y}_{k} \right) -\lambda_{n} I(f) \right\}\eqsp.
\] 
Lemma \ref{NPHMM:lem:EM} only ensures that $(\widehat{f}^t,\widehat{\nu}^t)_{t\ge 0}$ converges towards a local maximum of the penalized pseudo likelihood. This limitation is proper to models with hidden data. 
\end{rem}

\subsection{Experimental results}
\label{sec:exp:results}
This section illustrates the convergence of the estimates \eqref{NPHMM:eq:fn} using the EM procedure of Section \ref{NPHMM:sec:EMalgo}. The state-space is $\mathbb{X} = [0,1]$ and the unknown function $f_\star$ is given by
$$\begin{array}{ccccc}
\paramstar & : & [0,1]& \to & \Rset^{2} \\
& &x & \mapsto & \left(\cos(\pi x),\sin(\pi x)\right)\eqsp.
\end{array}$$
Therefore, throughout this section $m=1$ and $\ell=2$. As shown in Section~\ref{NPHMM:sec:identifiability}, the identifiability of $\paramstar$ up to an isometric function of $[0,1]$ can be obtained:
\begin{enumerate}[-]
\item In the case $b=1$ when $\nu_{1,\star}$ is assumed to be known.
\item In the case $b=2$ when $\mathcal{D}_2$ is the set of probability densities defined on $\mathbb{X}^2$ and of the form $\nu(x,x') = c(x)\rho(\|x-x'\|)$.
\end{enumerate}
The performance of the algorithm is assessed with two numerical experiments. 
\begin{enumerate}[-]
\item First, $(X_k)_{k\ge0}$ is assumed to be i.i.d. uniformly distributed on $[0,1]$ and only $\paramstar$ is estimated using $b=1$ in \eqref{NPHMM:eq:fn}. 
\item Then, $(X_k)_{k\ge0}$ is assumed to be a Markov chain with density kernel  given by
\[
q_\star(x,x') =q_{a_{\star}}(x,x') \eqdef C_{a_{\star}}(x)\mathrm{exp}\left(-\frac{\|x'-x\|}{a_{\star}}\right)
\]
and $a_\star$ and $\paramstar$ are estimated using $b=2$ in \eqref{NPHMM:eq:fn}. 
\end{enumerate}
In both cases, we wish to use the Sobolev setting of Section~\ref{NPHMM:sec:sobo} with $\lambda_n$ such that $\lambda_n \propto \log(n)n^{1/2}$ and  $I(f) = ||f||^{1/v}_{W^{2,2}}$ with  $1/v>b\ell=2b$ so that the hypothesis of Propositions \ref{NPHMM:consit:sobo1} and \ref{NPHMM:th:cons:sobo} are fulfilled. However, as discussed in the next section, such a complexity function $I$ may be intractable for the optimization problem.

\subsubsection{Approximations}
The computation of the intermediate quantities \eqref{NPHMM:eq:Q1} and \eqref{NPHMM:eq:Q2} requires an approximation of the conditional expectations $\mathbb{E}_{\widehat{\nu}^t,\param{}{t}}\left[ h(\mathbf{X}_k,\mathbf{Y}_k)\middle| \mathbf{Y}_{k}\right]$. For each $0\le k\le n-1$, the approximation of the distribution of $\mathbf{X}_k$ conditionally on $\mathbf{Y}_k$ when the parameters are $(\widehat{\nu}^t,\param{}{t})$ is dealt with Monte Carlo simulations. For each $t\ge0$ and each $0\le k\le n-1$, the Monte Carlo approximation is based on a set of particles $\{\mathbf{\Xi}^{t,j}_k\}_{j=1}^{N_{mc}}$, where \linebreak $\mathbf{\Xi}^{t,j}_k = (\mathbf{\xi}^{t,j}_{k,0},\ldots,\mathbf{\xi}^{t,j}_{k,b-1})$, associated with weights $\{\omega^{t,j}_k\}_{j=1}^{N_{mc}}$ such that for any bounded function $h$:
\[
\Esp_{\widehat{\nu}^t,\param{}{t}}\left[h(\mathbf{X}_k , \mathbf{Y}_k)\Big|\mathbf{Y}_k \right] \approx \sum_{j=1}^{N_{mc}} \omega_k^{t,j} h(\mathbf{\Xi}^{t,j}_k , \mathbf{Y}_k)\eqsp.
\]
Therefore, \eqref{NPHMM:eq:Q1} and \eqref{NPHMM:eq:Q2} are  approximated by:
\begin{align}
Q_t^1(\nu) &\approx \sum_{k=0}^{n-1} \sum_{j=1}^{N_{mc}} \omega_k^{t,j} \ln\left\{\nu(\mathbf{\Xi}^{t,j}_k)\right\}  \eqsp, \label{NPHMM:eq:nueq}\\
Q_t^2(f) &\approx -\frac{1}{2}\sum_{k=0}^{n-1} \sum_{j=1}^{N_{mc}}\omega_k^{t,j} \sum_{i=0}^{b-1}\| Y_{bk+i} - f(\xi^{t,j}_{k,i})\|^2 -\lambda_{n} \|f\|^{1/v}_{W^{2,2}}\eqsp.\label{NPHMM:eq:opti1}
\end{align}
However, the maximization of \eqref{NPHMM:eq:opti1} when $1/v >2b$ may be complex. Relaxing the hypothesis $1/v >2b$ by choosing $I(f) = \|f\|^{2}_{W^{2,2}}$ ($1/v=2$) allows to compute the maximizer  of \eqref{NPHMM:eq:opti1} as in \cite{Boor::1966} where the setting is similar except that $I(f) = \|f''\|^2_{L^2}$. \cite{Boor::1966} shows that the optimization problem can be written as an orthogonal projection in a Hilbert space. Nevertheless, using $1/v>2b$ (where $2b=2$ in the first study and $2b=4$ in the second one) as requested by Propositions \ref{NPHMM:consit:sobo1} and \ref{NPHMM:th:cons:sobo} leads to a much more complicated optimization problem since it can not be interpreted as an orthogonal projection in a Hilbert space. Moreover, the maximization of \eqref{NPHMM:eq:opti1} has been widely studied when $I(f) = \|f\|^{1/v}_{W^{2,2}}$ is replaced by $I(f) = \|f''\|^{2}_{L^2}$. In this setting, $\param{}{p+1}$ is then a regression spline (see for instance \cite{Boor::1966,hastie::1990}). Therefore, the constraints on $I(f)$ required by Propositions \ref{NPHMM:consit:sobo1} and \ref{NPHMM:th:cons:sobo} are relaxed in the simulations below where  $I(f) = \|f''\|^{2}_{L^2}$ and where pre-built optimized routines\footnote{In the following simulations, we use the csaps Matlab function from the Curve Fitting Toolbox to perform the M-step based on smoothing splines.} are used to compute $\param{}{t+1}$ given $\param{}{t}$.

\subsubsection{Experiment 1: $(X_k)_{k\ge 0}$ i.i.d.}
In this section, $b=1$ and $\nu_{1,\star}=1$ is assumed to be known. The estimation of $f_\star$  is performed with $N_{mc}=100$. In this case, for each $t\ge 0$, $0\le k \le n-1$ and $1\le j \le N_{mc}$,
\[
\mathbf{\xi}^{t,j}_{k,0} = \mathbf{\xi}^{t,j}_{k}\sim\nu_{1,\star}\quad\mbox{and}\quad \omega_{k}^{t,j}\propto \varphi(Y_{k} - \widehat{f}^t(\mathbf{\xi}^{t,j}_{k}))\eqsp.
\]
Figure~\ref{NPHMM:fig:boxplot} displays the $\rmL^2$ error of the estimation of $\paramstar$ after $100$ iterations as a function of the number of observations. The $\rmL^2$ estimation error decreases quickly for small values of $n$ (lower than $5000$) and then goes on decreasing at a lower rate as $n$ increases. It can be seen that even with a great number of observations, a small bias still remains for both functions (with a mean a bit lower than $0.05$). Indeed, there are always small errors in the estimation of $\paramstar$ around $x=0$ and $x=1$.

\begin{figure}[H]
\centering
\subfloat[$f_1$.]{\includegraphics[scale=.35]{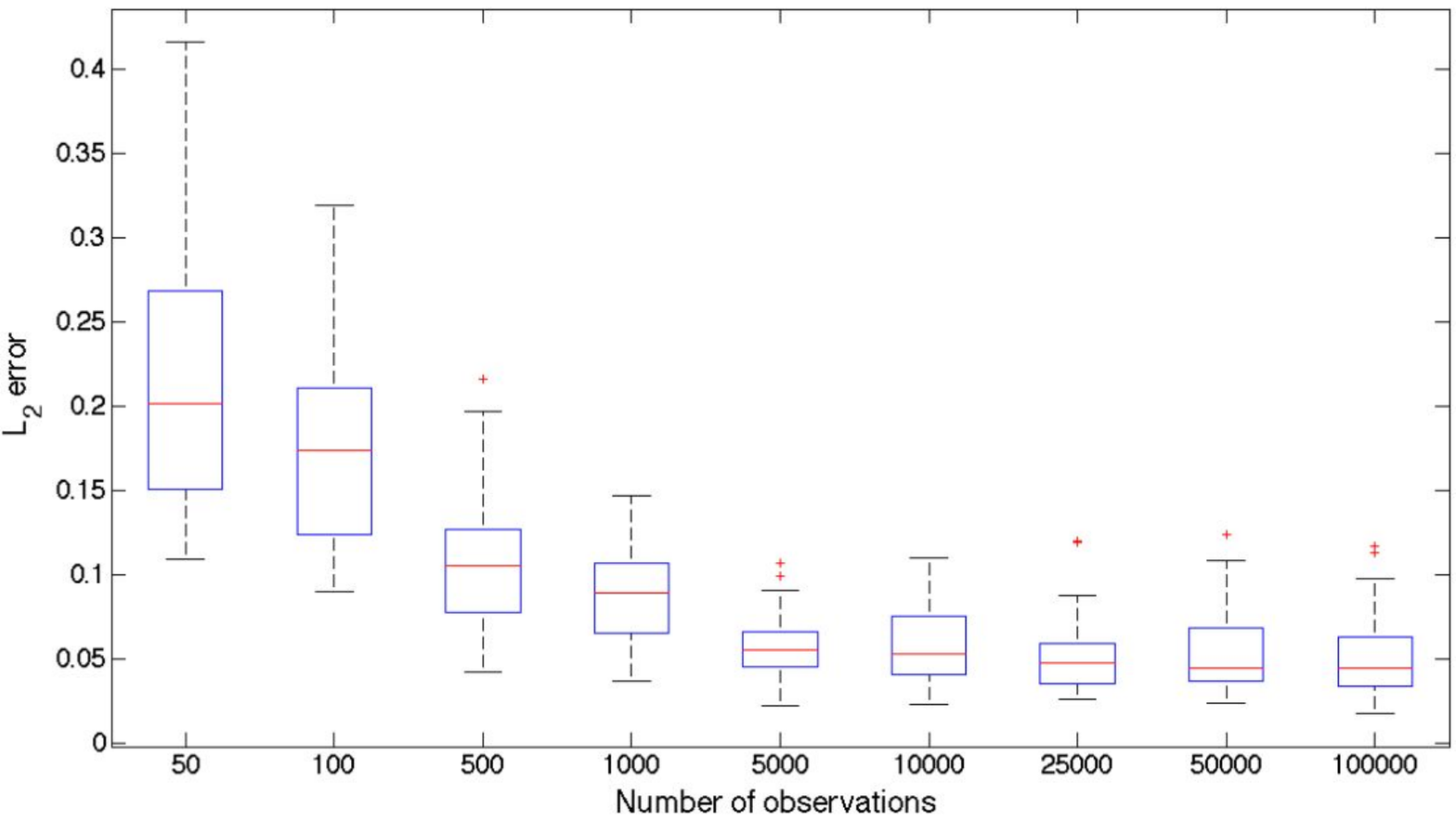}}\\
\subfloat[$f_2$.]{\includegraphics[scale=.35]{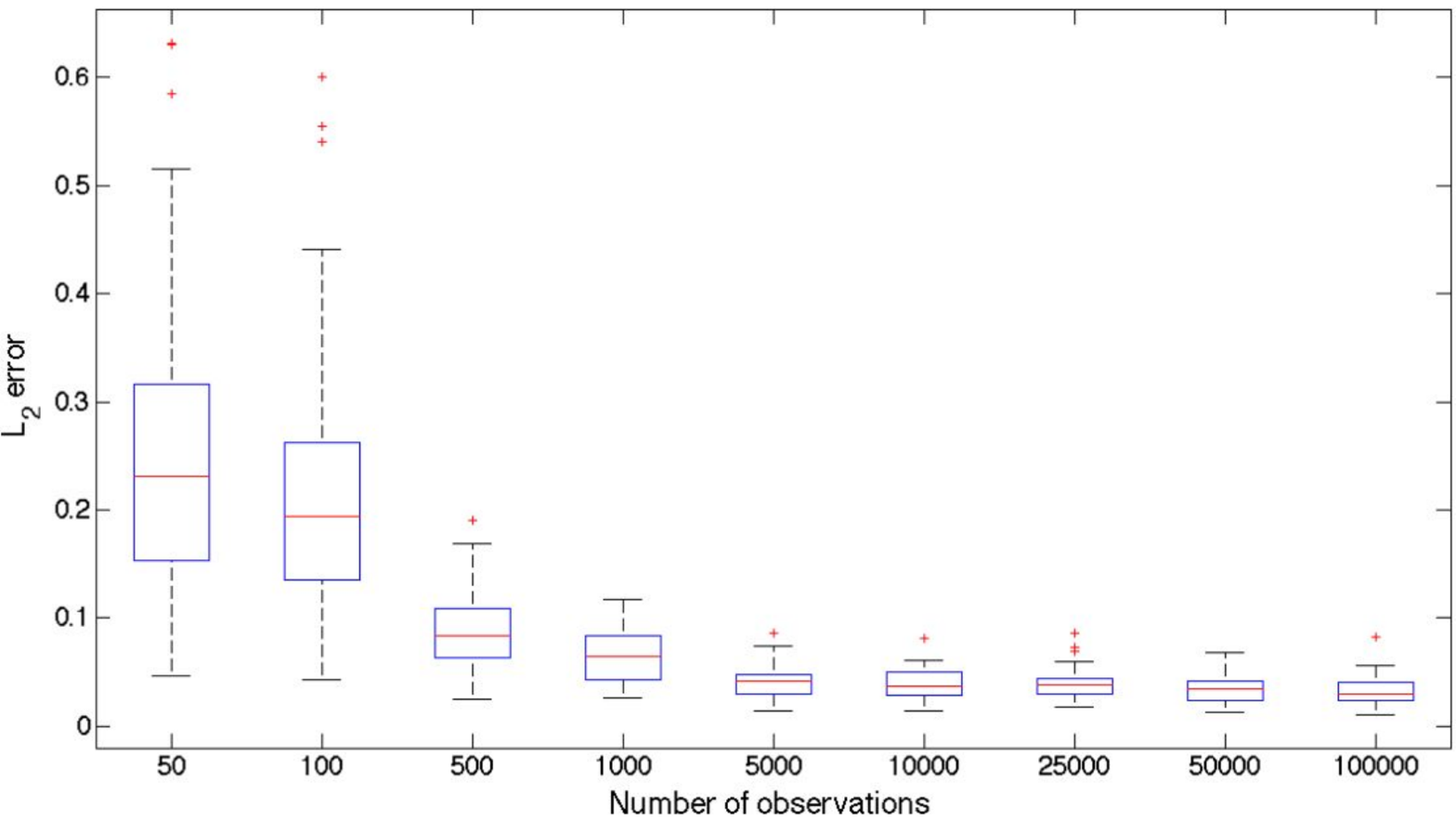}}
\caption{$\rmL^2$ error after $100$ iterations over $100$ Monte Carlo runs.}
\label{NPHMM:fig:boxplot}
\end{figure}
Figure~\ref{NPHMM:fig:finalf} shows the estimates after $100$ iterations when $n=25.000$. We observe on this Monte Carlo study that all the runs converge towards the isometric transformation $x\mapsto \paramstar(1-x)$.  This can be explained by the choice of the starting point of the EM algorithm.  The isometry is used in  Figure~\ref{NPHMM:fig:boxplot} to compute the $\rmL^2$ error. This simulation illustrates the identifiability results obtained in Section~\ref{NPHMM:sec:identifiability}.

\begin{figure}[H]
\centering
\subfloat[With no isometry for $f_1$.]{\includegraphics[scale=.3]{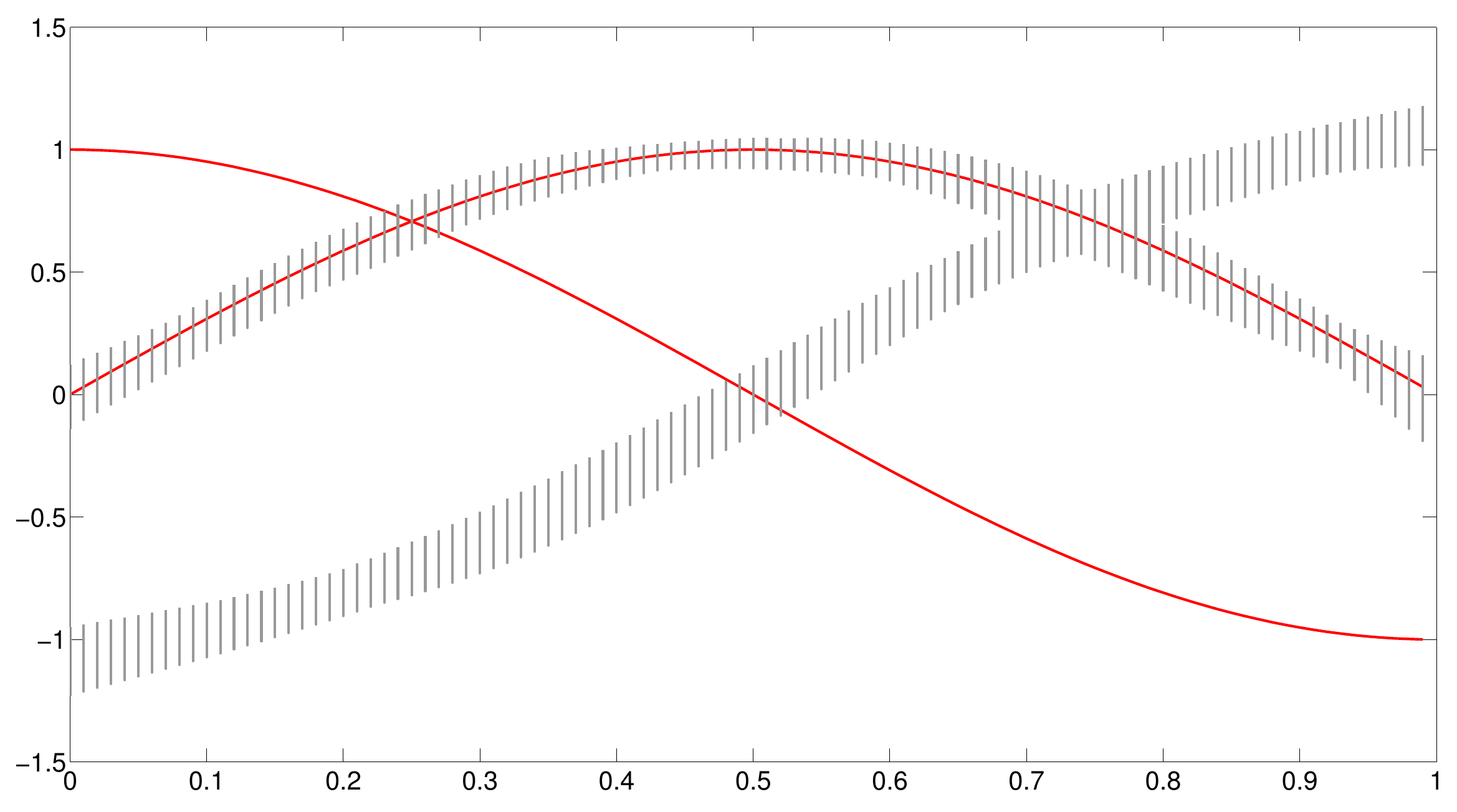}}\\
\subfloat[With the isometry $x\mapsto 1-x$ for $f_1$.]{\includegraphics[scale=.3]{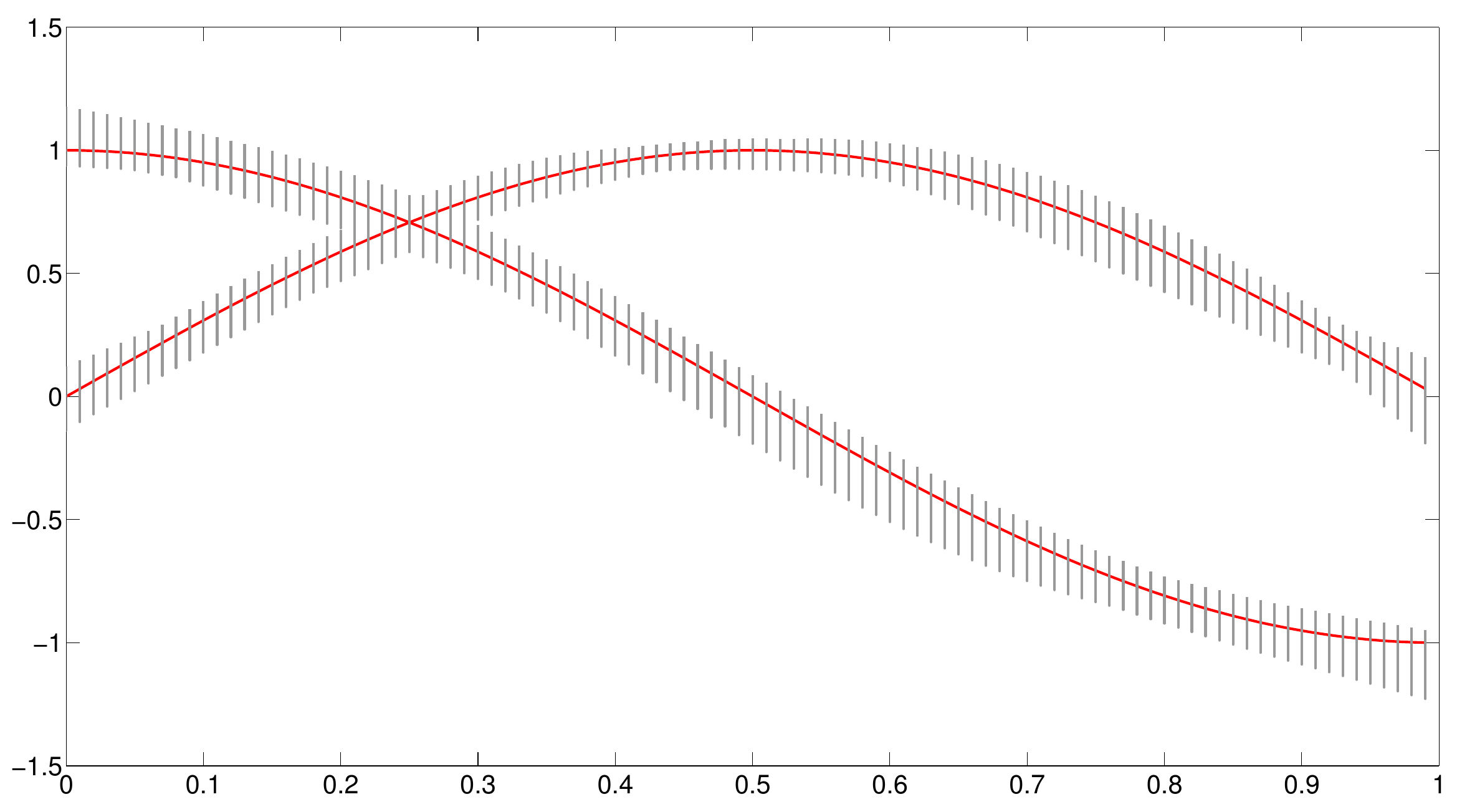}}
\caption{True functions (bold lines) and estimates after $100$ iterations (vertical lines)  over $100$ Monte Carlo runs ($n=25.000$).}
\label{NPHMM:fig:finalf}
\end{figure}

\subsubsection{Experiment 2: $(X_k)_{k\ge0}$ Markov chain}

In this section, $b = 2$ and  $a_{\star}$ and $\paramstar$ are estimated.
Define for any $a>0$, 
\begin{align*}
\nu_a(x,x') &= \nu_{1,a}(x)\cdot c_a(x)\exp\left(-\frac{|x-x'|}{a}\right)\eqsp,\\ 
\nu_{1,a}(x) &\propto  c_a^{-1}(x) = \int_{[0,1]} \exp\left(-\frac{|x-x'|}{a}\right)\mathrm{d}x'\eqsp.
\end{align*}
$\widehat{\nu}^{t+1}$ is given by  $\nu_{\widehat{a}^{t+1}}$ where $\widehat{a}^{t+1}$ is computed by maximizing the function 
\[
a\mapsto \log\left(a+a^2(\exp(-1/a)-1)\right) + \frac{1}{na}\sum_{k=0}^{n-1}\sum_{j=1}^{N_{mc}}\omega_k^{t,j} |\xi_{k,0}^{t,j}-\xi_{k,1}^{t,j}| \eqsp,
\]
where, for all $0\le k\le n-1$,  $(\xi_{k,0}^{t,j},\xi_{k,1}^{t,j})_{j=1}^{N_{mc}}$ are independently sampled uniformly in $[0,1]\times[0,1]$ and associated with the importance weights: 
\begin{equation}
\label{NPHMM:eq:omega}
\omega_k^{t,j}\propto\dens_{\widehat{a}^t}(\xi_{k,0}^{t,j})q_{\widehat{a}^t}(\xi_{k,0}^{t,j},\xi_{k,1}^{t,j})\denseps(Y_{2k}-\widehat{f}^t(\xi_{k,0}^{t,j}))\denseps(Y_{2k+1}-\widehat{f}^t(\xi_{k,1}^{t,j}))\eqsp.
\end{equation}
The Monte Carlo approximations are computed using $N_{mc} =200$ and $20.000$ observations ({\em i.e.} $n=10.000$) are sampled. Figure~\ref{NPHMM:fig:a} displays the estimation $a_{\star}$ as a function of the number of iterations of the EM algorithm over $50$ independent Monte Carlo runs. The estimates converge to the true value of $a_{\star}$ after few iterations (about $25$).
\begin{figure}[H]
\centering
\includegraphics[scale=.25]{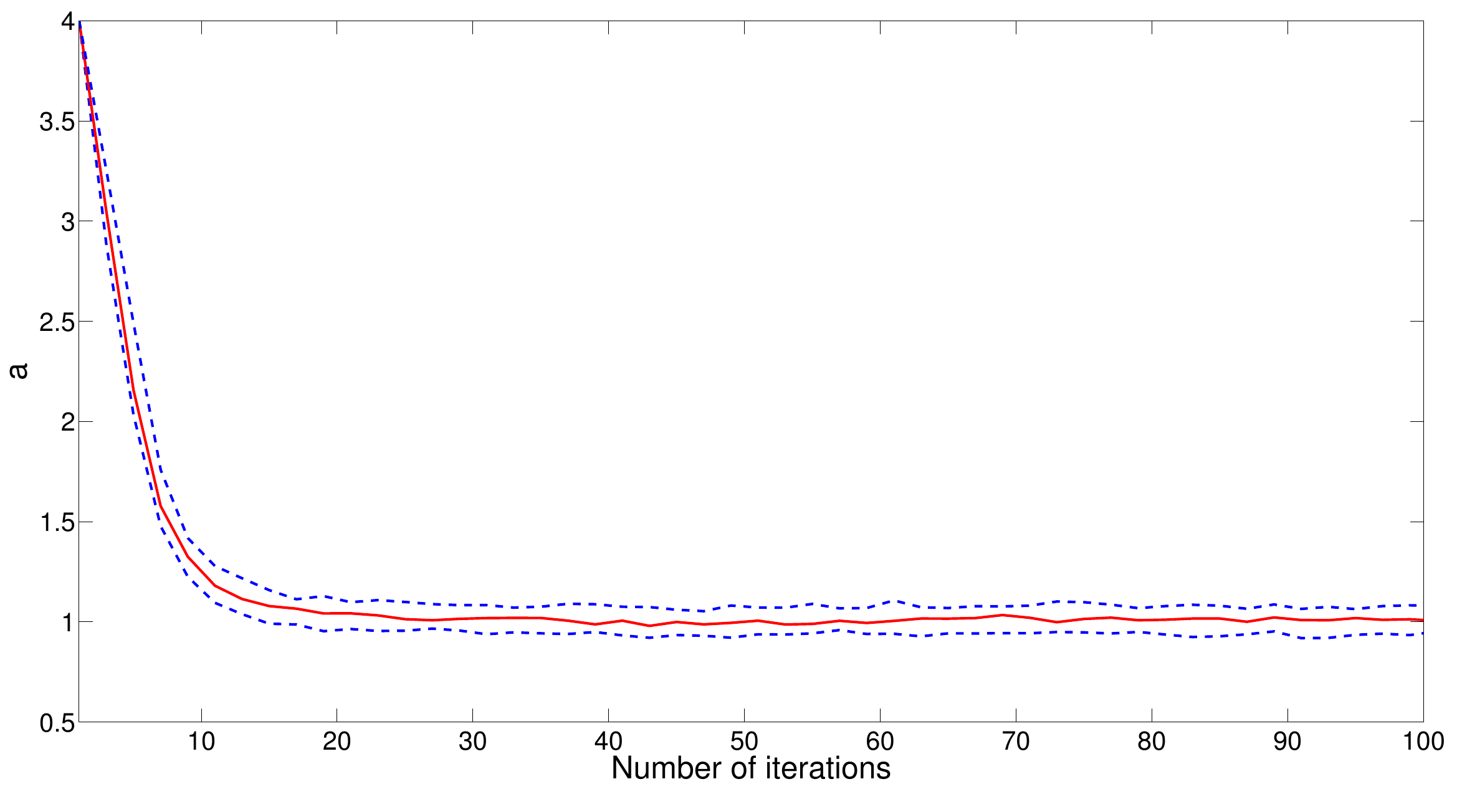}
\caption{Estimation of $a_{\star}$ as a function of the number of iterations of the EM algorithm. The true value is $a_{\star}=1$. Median (bold line) and upper and lower quartiles (dotted line) over $50$ Monte Carlo runs.}
\label{NPHMM:fig:a}
\end{figure}
Figure~\ref{NPHMM:fig:image} illustrates Corollary~\ref{NPHMM:cons:image}. It displays the estimation of $\paramstar([0,1])$ after $100$ iterations for several Monte Carlo runs. It shows that despite the variability of the estimation, the image is well estimated with few observations.
\begin{figure}[H]
\centering
\includegraphics[scale=.25]{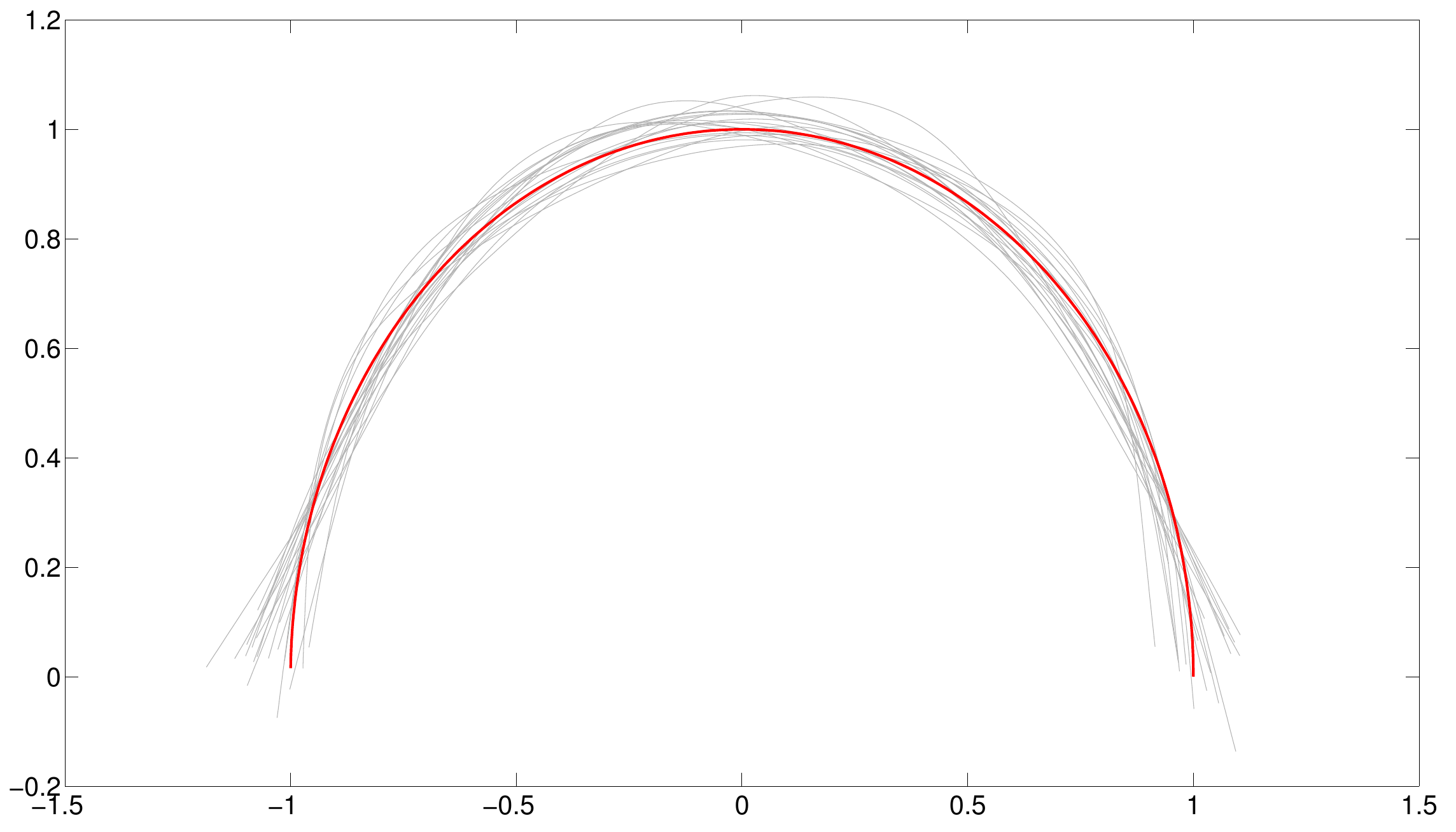}
\caption{True image $\paramstar([0,1])$ (red) and estimates after $100$ iterations of the algorithm over $100$ Monte Carlo runs (grey).}
\label{NPHMM:fig:image}
\end{figure}

\section{Proofs}
\label{NPHMM:sec:proofs}
\subsection{Proof of Proposition~\ref{NPHMM:prop:deviation:G}}
\label{sec:proof:NPHMM:prop:deviation:G}
Recall that for any probability density function $p$ on $\mathbb{R}^{b\dimY}$, $g_{p}$ is defined in \eqref{eq:defgp} by 
\begin{equation*}
g_{p} \eqdef \frac{1}{2} \ln \frac{p + p_{\star}}{2p_{\star}}\eqsp.
\end{equation*}
The proof relies on the application of Proposition~\ref{NPHMM:prop:concentration} and Proposition~\ref{NPHMM:prop:ineg:maximale} to obtain first a concentration inequality for the class of functions $\mathcal{G}_M$, where $M\ge1$, defined as:
\[
\mathcal{G}_M \eqdef \left\{g_{p_{f,\nu}} \; ;\; \dens\in\mathcal{D}_b,\; f\in\mathcal{F}\; \mbox{and}\;I(f)\le M \right\}\eqsp,
\]
where $p_{f,\nu}$ is defined by \eqref{NPHMM:eq:pf}. For any $p>0$, denote by $\rmL^{p}(\PPim)$ the set of functions $g:\mathbb{R}^{b\dimY}\to \mathbb{R}$ such that $\mathbb{E}\left[\left|g(\mathbf{Y}_0)\right|^{p}\right]<+\infty$. For any $\kappa>0$ and any set $\mathcal{G}$ of functions from $\mathbb{R}^{b\dimY}$ to $\mathbb{R}$, let $N_{[]}(\kappa,\mathcal{G},\norminf{\cdot}{\rmL^{p}(\PPim)})$ be the smallest integer $N$ such that there exists a set of functions $\left\{\left(g^{L}_{i} ,g^{U}_{i} \right)\right\}_{i=1}^{N}$  for which: 
\begin{enumerate}[a)]
\item $\norminf{g^{U}_{i} - g^{L}_{i}}{\rmL^{p}(\PPim)} \le \kappa$ for all $i\in\{1,\cdots,N\}$;
\item for any $g$ in $\mathcal{G}$, there exists $i\in\left\{1,\cdots,N \right\}$ such that 
\[
g_{i}^{L} \le g \le g_{i}^{U}\eqsp.
\]
\end{enumerate}
$N_{[]}(\kappa,\mathcal{G},\norminf{\cdot}{\rmL^{p}(\PPim)})$ is the $\kappa$-number with bracketing of $\mathcal{G}$, and $H_{[]}(\kappa,\mathcal{G},\norminf{\cdot}{\rmL^{p}(\PPim)}) \eqdef \ln N_{[]}(\kappa,\mathcal{G},\norminf{\cdot}{\rmL^{p}(\PPim)}) $  is the  $\kappa$-entropy with bracketing of $\mathcal{G}$. 
For any bounded function $g$, define
\begin{equation}
\label{NPHMM:eq:defSn}
S_{n}(g) \eqdef n \int g \,\rmd(\PP_{n} - \PPim) = \sum_{k=0}^{n-1}g(\bfY_{k}) - n \mathbb{E}[g(\bfY_0)]\eqsp.
\end{equation}

\paragraph{Application of Proposition~\ref{NPHMM:prop:concentration}}
Proposition~\ref{NPHMM:prop:concentration} is applied to the class of functions $\overline{\mathcal{G}}_M$ defined as
\[
\overline{\mathcal{G}}_M\eqdef\left\{g-\mathbb{E}\left[g(\bfY_0)\right];\; g\in\mathcal{G}_M\right\}\eqsp.
\]
\begin{enumerate}[-]
\item By H\ref{assum:pen}, there exists $C>0$ such that  for any $i\ge0$, and any $g\in \mathcal{G}_M$,
\begin{align*}
\left|g\left( \mathbf{Y}_i\right)\right|\le  C M^\upsilon \left( 1+\|\mathbf{Y}_i\|\right)&\le C M^\upsilon \left( 1+||\mathbf{\paramstar}(\mathbf{X}_i)\| + \|\boldsymbol{\epsilon}_i\|\right)\eqsp,\\
&\le C M^\upsilon \left( 1+ \|\paramstar\|_{\infty}  + \|\boldsymbol{\epsilon}_i\|\right)\eqsp,\\
&\le C M^\upsilon \left( 1+  \|\boldsymbol{\epsilon}_i\|\right)\eqsp.
\end{align*}
Define $\mathbf{U}_i \eqdef C M^\upsilon \left( 1 + \|\boldsymbol{\epsilon}_i\|\right)$. Then, the random variables  $(\mathbf{U}_i)_{i\ge 0 }$ are i.i.d. and for all $i\ge 0$,  $|g\left( \mathbf{Y}_i\right)-\mathbb{E}\left[g(\bfY_0)\right]|\le \mathbf{U}_i + \mathbb{E}\left[\mathbf{U}_0\right]$. Furthermore, 
\begin{equation*}
\mathbb{E}\left[ (\mathbf{U}_i + \mathbb{E}\left[\mathbf{U}_0\right])^{2k}\right] \le k!\nu c^{k-1}\quad \mbox{with}\quad \nu\eqdef CM^{2\upsilon}\quad\mbox{and}\quad c \eqdef CM^{2\upsilon}\eqsp.
\end{equation*}
\item On the other hand, since the random variables $(\boldsymbol{\epsilon}_k)_{k\ge0}$ are i.i.d. and $(\mathbf{X}_k)_{k\ge 0}$ is $\Phi$-mixing, $(\mathbf{Y}_k)_{k\ge0}$ is also $\Phi$-mixing with mixing coefficients $(\phi^{\mathbf{Y}}_i)_{i\ge 0} $ satisfying, for all $i\ge 1$, $\phi^{\mathbf{Y}}_i \le \phi^{\mathbf{X}}_i =\phi^{X}_{(i-1)b+1}$. Therefore $\mathbf{\Phi^{Y}} = \sum_{i\ge 1}(\phi_i^{\mathbf{Y}})^{1/2}<\infty$.
\end{enumerate}
By Proposition~\ref{NPHMM:prop:concentration}, there exists a positive constant $C$ such that for any positive $x$, 
\begin{equation}
\label{NPHMM:eq:concentration:first}
\mathbb{P}\left[\sup_{g\in\mathcal{G}_M} \left|S_n(g) \right| \ge \mathbb{E}\left[ \sup_{g\in\mathcal{G}_M} \left|S_n(g) \right|\right] + C\mathbf{\Phi^Y}\times\left(\sqrt{nx} + x\right)M^{\upsilon}\right] \le \mathrm{e}^{-x}\eqsp.
\end{equation}

\paragraph{Application of Proposition~\ref{NPHMM:prop:ineg:maximale}}
Proposition~\ref{NPHMM:prop:ineg:maximale} is used to control the inner expectation in \eqref{NPHMM:eq:concentration:first}. 
Let $r>1$. By \cite[Lemma 7.26]{massart:2007} and since the Hellinger distance is bounded by $1$,  there exists a constant $\delta$ such that for any $g = g_{p_{f,\dens}}\in\mathcal{G}_M$. 
\[
\|g\|^{2r}_{\mathrm{L}^{2r}(\PPim)} \le \delta\eqsp.
\]
By Lemma~\ref{NPHMM:prop:bracket:GM}, for any $q> 1$,  any $s>b\dimY/q$ and any $\beta>s + b\dimY(1-1/q)$, there exists a constant $c$ such that, for all $u>0$,
\begin{equation}\label{NPHMM:eq:brackets}
H_{[]}(u,\|\cdot\|_{\mathrm{L}^{2r}(\PPim)},\mathcal{G}_M) \le  c \left(\frac{M^{\upsilon(s+\beta+b\dimY/q)}}{u^{2r}} \right)^{b\dimY/s}
\end{equation}
and
\[
\varphi(\delta) \eqdef \int_{0}^{\delta} H_{[]}^{1/2}(u,\|\cdot\|_{\mathrm{L}^{2r}(\PPim)},\mathcal{G}_M)\mathrm{d}u \le cM^{(s+\beta+b\dimY/q)b\dimY\upsilon/(2 s)}  \int_{0}^{\delta}  u^{-rb\dimY/s}  \rmd u\eqsp.
\]
Choosing $\beta\le  s + b\dimY(1-1/q) +2$,  if $s$ goes to $+\infty$ then the last integral is finite, and   $(s+\beta+b\dimY/q)b\dimY\upsilon/(2s)$ converges to $b\dimY\upsilon$,  so that for any $\eta>0$ there exists a positive constant $c$ such that
\begin{equation*}
\varphi(\delta)\le cM^{b\dimY\upsilon + \eta}\eqsp.
\end{equation*} 
Finally, by Proposition~\ref{NPHMM:prop:ineg:maximale} for any $\eta>0$, there exists a constant $A$ such that for $n$ large enough
\[
\mathbb{E}\left[ \sup_{g\in\mathcal{G}_M} \left|S_n(g) \right|\right]\le A\sqrt{n}M^{b\dimY\upsilon + \eta}\eqsp.
\]
Then, by \eqref{NPHMM:eq:concentration:first}, this yields
\begin{equation}
\label{NPHMM:eq:concentration:GM}
\mathbb{P}\left[\sup_{g\in\mathcal{G}_M}| S_n (g)| \ge   c\mathbf{\Phi^{Y}}\times\left(\sqrt{nx} + x \right)M^{\upsilon}+ A\sqrt{n}M^{b\dimY\upsilon + \eta} \right] \le \mathrm{e}^{-x}\eqsp.
\end{equation}
Proposition~\ref{NPHMM:prop:deviation:G} is then proved using a peeling argument. By \eqref{NPHMM:eq:defSn} and \eqref{NPHMM:eq:concentration:GM}, for any $M\ge1$, any large enough $n$ and any $x>0$, if $\gamma=b\dimY\upsilon + \eta$,
\begin{equation}
\label{NPHMM:eq:dev:GM2}
 \mathbb{P}\left[\sup_{g\in\mathcal{G}_M} \frac{\left|\int g \ \rmd(\PP_n - \PPim)\right| }{M^{\gamma}}\ge   c\mathbf{\Phi^{Y}}\times\left(\sqrt{\frac{x} {n}}+ \frac{x}{n}\right)+ \frac{A}{\sqrt{n}} \right] \le e^{-M^{\gamma-\upsilon} x}\eqsp.
\end{equation}
We can write
\[
\mathbb{P}\left[ \sup_{f\in\mathcal{F}, \ \dens\in\mathcal{D}_b} \frac{\left|\int g_{p_{f,\dens}} \ \rmd(\PP_n - \PPim)\right| }{1\vee I(f)^{\gamma}}\ge   c\mathbf{\Phi^{Y}}\times\left(\sqrt{\frac{x} {n}}+ \frac{x}{n}\right)+ \frac{2^\gamma A}{\sqrt{n}} \right]\le P_1 + \sum_{k=0}^{+\infty}T_k\eqsp,
\]
where
\begin{align*}
P_1&\eqdef\mathbb{P} \left[  \sup_{\substack{f\in\mathcal{F};\;I(f) \le 1,\\ \dens\in\mathcal{D}_b}}  \frac{\left|\int g_{p_{f,\dens}} \ \rmd(\PP_n - \PPim)\right| }{1\vee I(f)^{\gamma}}\ge  c\mathbf{\Phi^{Y}}\times\left(\sqrt{\frac{x}{n}}+ \frac{x}{n}\right)+ \frac{2^\gamma A}{\sqrt{n}} \right]\eqsp,\\
T_k&\eqdef\mathbb{P} \left[ \sup_{\substack{f\in\mathcal{F};\; 2^k <I(f) \le 2^{k+1},\\ \dens\in\mathcal{D}_b}} \frac{\left|\int g_{p_{f,\dens}} \ \rmd(\PP_n - \PPim)\right| }{1\vee I(f)^{\gamma}}\ge   c\mathbf{\Phi^{Y}}\times\left(\sqrt{\frac{x} {n}}+ \frac{x}{n}\right)+ \frac{2^\gamma A}{\sqrt{n}} \right]\eqsp.
\end{align*}
By \eqref{NPHMM:eq:dev:GM2},
\begin{multline*}
P_1\le \mathbb{P} \left[  \sup_{g\in\mathcal{G}_1}  \left|\int g  \ \rmd(\PP_n - \PPim)\right| \ge   c\mathbf{\Phi^{Y}}\times\left(\sqrt{\frac{x}{n}}+ \frac{x}{n}\right)+ \frac{2^\gamma A}{\sqrt{n}} \right]\eqsp,\\
\le  \mathbb{P} \left[  \sup_{g\in\mathcal{G}_1}  \left|\int g  \ \rmd(\PP_n - \PPim)\right| \ge   c\mathbf{\Phi^{Y}}\times\left(\sqrt{\frac{x}{n}}+ \frac{\sqrt{c}x }{n}\right)+ \frac{A}{\sqrt{n}} \right]\le \mathrm{e}^{-x}
\end{multline*}
and for all $k\ge 0$,
\begin{multline*}
T_k\le \mathbb{P} \left[ \sup_{g\in\mathcal{G}_{2^{k+1}}} \frac{\left|\int g  \ \rmd(\PP_n - \PPim)\right| }{2^{\gamma(k+1)}}\ge \frac{c}{2^\gamma} \mathbf{\Phi^{Y}}\times\left(\sqrt{\frac{x}{n}}+ \frac{x}{n}\right)+ \frac{A}{\sqrt{n}} \right]\eqsp,\\
\le  \mathbb{P} \left[ \sup_{g\in\mathcal{G}_{2^{k+1}}} \frac{\left|\int g  \ \rmd(\PP_n - \PPim)\right| }{2^{\gamma(k+1)}}\ge c\mathbf{\Phi^{Y}}\times\left(\sqrt{\frac{x}{2^{2\gamma}n}}+ \frac{x}{2^{2\gamma}n}\right)+ \frac{A}{\sqrt{n}} \right]\le \mathrm{e}^{-2^{(\gamma-\upsilon)(k+1)}  x/2^{2\gamma}}\eqsp.
\end{multline*}
Using \eqref{NPHMM:eq:dev:GM2},
\begin{align*}
\mathbb{P}\bigg[ \sup_{f\in\mathcal{F}, \ \dens\in\mathcal{D}_b} \frac{\left|\int g_{p_{f,\dens}} \ \rmd(\PP_n - \PPim)\right| }{1\vee I(f)^{\gamma}}&\ge   c\mathbf{\Phi^{Y}}\times\left(\sqrt{\frac{x}{n}}+ \frac{x}{n}\right)+ \frac{2^\gamma A}{\sqrt{n}} \bigg]\\
&\le e^{-x} + \sum_{k=0}^{\infty} e^{-2^{(\gamma-\upsilon)(k+1)}  x/2^{2\gamma}}\\
&\le e^{-x} + \sum_{k=0}^{\infty}   e^{-(k+1) x\log (2)(\gamma-\upsilon)/2^{2\gamma} } \\
&\le e^{-x} + \frac{e^{- \alpha x }}{1-e^{- \alpha x }}\eqsp,   
\end{align*}
which concludes the proof of Proposition \ref{NPHMM:prop:deviation:G}.

\subsection{Proof of Proposition \ref{NPHMM:prop:nu:phi}}
\label{NPHMM:sec:proofs:prop:nu:phi}
Assume that $h(p_{f,\nu},p_{\paramstar,\nu_{1,\star}} )= 0$ (the proof of the  converse proposition is straightforward). Let $X'_0$ be a random variable on $\mathbb{X}$ with distribution $\nu(x)\mu(\rmd x)$. Since $\epsilon_0$ is a Gaussian random variable, $h(p_{f,\nu},p_{\paramstar,\nu_{1,\star}} )= 0$ implies that $f(X'_0)$ has the same distribution as $\paramstar(X_0)$.  

\paragraph{Proof that $f$ and $\paramstar$ have the same image in $\mathbb{R}^\dimY$.}
Let $y\in f(\mathbb{X})$, $n\ge1$ and $B(y,n^{-1})$ be the open Euclidean ball in $\mathbb{R}^\ell$ centered at $y$ with radius $n^{-1}$.  As  $y\in f(\mathbb{X})$ and  $f$ is continuous, there exists a nonempty open subset $\mathcal{O}$ of $\mathbb{R}^m$ such that $ f^{-1}(B(y,n^{-1})) = \mathcal{O} \cap \mathbb{X}$.  Since $\overline{\overset{\circ}{\mathbb{X}}} = \mathbb{X}$,  $\overset{\circ}{\mathbb{X}}$ is not empty and so is the interior of $ f^{-1}(B(y,n^{-1}))$ (which is equal to $ \mathcal{O} \cap \overset{\circ}{\mathbb{X}}$). Therefore, $\mu\left\{ f^{-1}\left(B\left(y,n^{-1}\right)\right)\right\}>0$.  Then, using that $\dens\ge \dens_-$ and that $f(X'_0)$ has the same distribution as $\paramstar(X_0)$,  
\[
\mathbb{P}\!\left\{X_0\in \paramstar^{-1}\left(B\left(y,n^{-1}\right)\right)\right\} \! =\mathbb{P}\!\left\{X'_0\in f^{-1}\left(B\left(y,n^{-1}\right)\right)\right\} \! \ge \dens_-\, \mu\left\{ f^{-1}\left(B\left(y,n^{-1}\right)\right)\right\}>0\eqsp,
\]
 Hence, $\paramstar^{-1}\left(B\left(y,n^{-1}\right)\right)$ is nonempty and for all $n\ge 1$, there exists $x_n\in\mathbb{X}$ such that $\|y-\paramstar(x_n)\|<n^{-1}$. Moreover, for all $n\ge 1$, $\paramstar(x_n)$ lies in the compact set $\paramstar(\mathbb{X})$. This implies that $y\in \paramstar(\mathbb{X})$. The proof of the converse inclusion follows the same lines.

\paragraph{Proof that $\phi$ is bijective.}
Since $f(X'_0)$ has the same distribution as $\paramstar(X_0)$, $X_0$ has the same distribution as $\phi(X'_0)$ where $\phi \eqdef \paramstar^{-1} \circ f $. By H\ref{assum:fstar} $\phi$ exists and is $\mathcal{C}^1$. We prove that $|J_{\phi}|>0$ using the following result due to \cite[Theorem 2, p.99]{evans:gariepy:1992}.
\begin{lemma}
\label{NPHMM:lem:area}
If $\phi: \mathbb{X}\to \mathbb{X}$ is Lipschitz then, for any integrable function $g$,
\[
\int_\mathbb{X}g(x)\left|J_{\phi}(x)\right|\mu(\rmd x) = \int_\mathbb{X}\sum_{x\in\phi^{-1}(\{y\})}g(x)\mu(\rmd y)\eqsp.
\]
\end{lemma}
Define $A \eqdef \left\{ x\in \mathbb{X}\eqsp;\;  \forall x' \in \phi^{-1 }(\{x\}),\ |J_{\phi}(x')|>0 \right\}$.  Let $h_{1}$ be a bounded measurable real function on $\mathbb{X}$ and define $h \eqdef \mathds{1}_{A}h_1$. 
By Lemma~\ref{NPHMM:lem:area},
\begin{align*}
\mathbb{E}\left[h\circ \phi (X'_0)\right] &= \int_\mathbb{X} h_{1}(\phi(x') ) \mathds{1}_{A}(\phi(x')) \nu(x') \mu(\mathrm{d}x')\eqsp,\\
& =  \int_\mathbb{X} h_{1}(\phi(x') ) \mathds{1}_{A}(\phi(x')) \frac{\nu(x') }{|J_{\phi}(x')|} |J_{\phi}(x')| \mu(\mathrm{d}x')\eqsp,\\
& = \int_\mathbb{X} h_{1}(x)  \mathds{1}_{A}(x) \sum_{x'\in\phi^{-1}(\{x\})} \frac{\nu(x') }{|J_{\phi}(x')|}  \mu(\mathrm{d}x)\eqsp.
\end{align*}
Since $X_0$ has the same distribution as $\phi(X'_0)$,
\[
\int_\mathbb{X} h_{1}(x)\mathds{1}_{A}(x)  \dens_{1,\star}(x)\mu(\mathrm{d}x) = \int_\mathbb{X} h_{1}(x)  \mathds{1}_{A}(x) \sum_{x'\in\phi^{-1}(\{x\})} \frac{\nu(x') }{|J_{\phi}(x')|}  \mu(\mathrm{d}x)\eqsp.
\]
Applying Lemma~\ref{NPHMM:lem:area} with $g \eqdef  \mathds{1}_{|J_\phi|=0}$ implies that  $\mathds{1}_{A} = 1$ $\mu$-a.s. in $\mathbb{X}$ and, $\mu$-a.s.,
\begin{equation}
\label{NPHMM:eq:denstar=dens}
\dens_{1,\star}(x)= \sum_{x'\in\phi^{-1}(\{x\})} \frac{\dens(x') }{|J_{\phi}(x')|}\eqsp.
\end{equation}
Therefore, for $\mu$ almost every $x\in\mathbb{X}$ and for all $x'\in\phi^{-1}(\{x\})$,
\[
|J_{\phi}(x')|\ge\frac{\dens_-}{\dens_+}\eqsp.
\]
 By continuity of $J_\phi$ and using that $\overline{\overset{\circ}{\mathbb{X}}} = \mathbb{X}$, $|J_{\phi}(x)|>0$ for all $x\in\mathbb{X}$. Therefore, $\phi$ is locally invertible and, since $\mathbb{X}$ is compact, simply connected and arcwise connected, $\phi$ is bijective by \cite[Theorem $1.8$, p.$47$]{ambrosetti:prodi:1995}. Then \eqref{NPHMM:eq:denstar=dens} ensures that for $\mu$ almost every $x\in\mathbb{X}$, 
\begin{equation*}
\nu_{1,\star}(\phi(x))   =   \frac{\nu(x) }{|J_{\phi}(x)|} \eqsp, 
\end{equation*}
which concludes the proof of Proposition \ref{NPHMM:prop:nu:phi}.

\subsection{Proof of Proposition~\ref{NPHMM:prop:ident:Markov} and Corollary~\ref{NPHMM:cor:markov}}
\label{NPHMM:sec:proof:ident:Markov}
\paragraph{Proof of Proposition~\ref{NPHMM:prop:ident:Markov}}
The proof of \eqref{NPHMM:eq:Q:phi} follows the same lines as the proof of Proposition~\ref{NPHMM:prop:nu:phi}. Let $(X'_0,X'_1)$ be a random variable on $\mathbb{X}^2$ with probability density $\nu(x)q(x,x')$ on $\mathbb{X}^2$.  $h(p_{f,\nu^2},p_{\paramstar,\nu_{2,\star}} )= 0$ implies that $h(p_{f,\nu },p_{\paramstar,\nu_{\star} } )= 0$ and, by Proposition~\ref{NPHMM:prop:nu:phi}, $f(\mathbb{X}) = \paramstar(\mathbb{X})$  and $\phi = \paramstar^{-1} \circ f $ is bijective. Moreover, since $(\epsilon_0,\epsilon_1)$ has a Gaussian distribution, $h(p_{f,\nu^2},p_{\paramstar,\nu_{\star}^2} )= 0$ implies that $(\phi(X'_0),\phi(X'_1))$ has the same distribution as $ (X_0,X_1)$ so that for any $x$ in $\mathbb{X}$ and any bounded measurable function $f$ on $\mathbb{X}$, 
\[
\mathbb{E}\left[\phi(X'_1)\middle|X'_0 = \phi^{-1}(x)\right] =  \mathbb{E}\left[X_1\middle|X_0 = x\right]\eqsp.
\]
Following the proof of Proposition~\ref{NPHMM:prop:nu:phi}, this gives \eqref{NPHMM:eq:Q:phi}.

\paragraph{Proof of Corollary~\ref{NPHMM:cor:markov}}
Assume that 
\[
q_\star(x,x') = c_\star(x) \rho_{\star}(||x-x'||) \quad \mbox{and} \quad q (x,x') = c(x) \rho(||x-x'||)\eqsp.
\]
By \eqref{NPHMM:eq:Q:phi}, 
\begin{equation}
\label{NPHMM:eq:Q:phi:2}
c(x)\rho(||x-x'||)= |J_{\phi}(x')| c_\star(\phi(x)) \rho_{\star}(||\phi(x)-\phi(x')||)\;.
\end{equation} 
Applying \eqref{NPHMM:eq:Q:phi:2} with $x=x'$ implies $|J_{\phi}(x)| = \frac{\rho(0)}{\rho_{\star}(0)}\frac{c(x)}{c_\star(\phi(x))}$. Therefore, 
\begin{equation*} 
\frac{|J_{\phi}(x)|}{|J_{\phi}(x')|}= \frac{c(x)c_\star(\phi(x'))}{c(x')c_\star(\phi(x))} =\frac{|J_{\phi}(x')|}{|J_{\phi}(x)|}
\end{equation*} 
and then, there exists a constant $C$ such that for all $x\in\mathbb{X}$, $|J_{\phi}(x)|=C$. As $\phi$ is bijective we may write
\[
\mu(\mathbb{X}) = \mu(\phi(\mathbb{X})) = \int_{\phi(\mathbb{X})}\mu(\rmd x) = \int_{\mathbb{X}}|J_{\phi}(x)|\mu(\rmd x) = C \mu(\mathbb{X})\eqsp,
\]
which leads to $C=1$ since $0<\mu(\mathbb{X})<\infty$. By \eqref{NPHMM:eq:Q:phi:2}, for any $x$ and $x'$ in $ \mathbb{X}$,
\begin{equation}
\label{NPHMM:eq:q:qstar}
\rho(||x-x'||)  = \rho_{\star}(||\phi(x) - \phi(x')||).
\end{equation}
Let $x_0\in\overset{\circ}{\mathbb{X}}$, $y_0 = \phi(x_0)$ and $d_0,d_0'>0$ be such that $B(x_0,d_0) \eqdef \{x\in\mathbb{R}^m\;,\; ||x_0 - x||< d_0\} \subset \mathbb{X}$  and $\phi(B(x_0,d_0)) \subset B(y_0,d_0')$.

 Let $d<d_0$ and denote by $S(x_0,d)$ the set  $S(x_0,d)\eqdef \{x\in\mathbb{R}^m\;,\; ||x_0 - x ||=d\}$. As $\rho_\star$ is one-to-one, write $F = \rho_\star^{-1} \circ \rho$. \eqref{NPHMM:eq:q:qstar} implies that  
 $\phi(S(x_0,d)) \subset S(y_0,F(d))$. Furthermore, using the compactness and the connectivity of $S(x_0,d)$,  $\phi(S(x_0,d)) = S(y_0,F(d))$ which, together with the continuity of $\phi$, guarantees that $\phi(B(x_0,d)) = B(y_0,F(d))$. Finally, because $\phi$ preserves the volumes, for any $d<d_0$, $F(d) = d$ and for any $x \in\mathbb{X}$ and any $x'\in B(x,d_0)$, $||x -x'|| =||\phi(x) - \phi(x')||$.  The proof is concluded using the connectivity of $\mathbb{X}$.

\appendix
\section{Concentration results for the empirical process of unbounded functions}
\label{NPHMM:sec:concentration}

Proposition \ref{NPHMM:prop:concentration} provides a concentration inequality on the empirical process over a class of functions $\mathcal{G}$ for which $|g(Z_i)|$ can be bounded uniformly in $g\in\mathcal{G}$  by an independent process $U_i$ with bounded moments. This unusual condition is more general than \cite[Theorem 3]{samson:2000} which considered a uniformly bounded class of functions.
\begin{proposition}\label{NPHMM:prop:concentration}
Let $(Z_n)_{n\ge 0}$ be a $\Phi$-mixing process taking values in a set $\mathcal{Z}$. Assume that the $\Phi$-mixing coefficients associated with  $(Z_n)_{n\ge 0}$ satisfy:
\[
\mathbf{\Phi } \eqdef \sum_{i=1}^{\infty}\phi_{i}^{1/2} <\infty\eqsp.
\]
Let $\mathcal{G}$ be some countable class of real valued measurable functions defined  on $\mathcal{Z}$.  Assume that there exists a sequence of independent random variables $(U_i)_{i\ge 0}$ such that:
\begin{enumerate}[-]
\item for any $g$ in $\mathcal{G}$ ,
\begin{equation}
\label{NPHMM:assum:G}
|g(Z_i) |\le U_i \ a.s.\eqsp;
\end{equation}
\item there exists some positive numbers $\nu$ and $c$ such that,
  for any $k\ge 1$:
\begin{equation}
\label{NPHMM:assum:Ui}
\sum_{i=0}^{n-1} \mathbb{E}\left[U_i^{2k}\right] \le k!n\nu c^{k-1}\eqsp.
\end{equation}
\end{enumerate}
Then,  for any positive $x$, 
\[
\mathbb{P}\left[S_n \ge 2\mathbf{\Phi }\times\left( 2\sqrt{n\nu x} + \sqrt{c}x \right)\right] \le \mathrm{e}^{-x}\eqsp,\] 
where
\[
S_n = \sup_{g\in\mathcal{G}} \left|\sum_{i=0}^{n-1}g(Z_i) \right|-\mathbb{E}\left[ \sup_{g\in\mathcal{G}} \left|\sum_{i=0}^{n-1}g(Z_i) \right|\right]\eqsp.
\]
\end{proposition}
\begin{proof}
For any real valued random variable and for any real random variable $X$, define $\psi_{X}(\lambda) \eqdef \ln\left( \mathbb{E}\left[\exp\left(\lambda X\right) \right] \right) $,
Following the proof of \cite[Theorem 3]{samson:2000} together with the discussion about the dependence structure in \cite[Section 2]{samson:2000}, we have
\begin{equation}
\exp\left(\psi_{S_n}\left(\frac{\lambda}{4}\right)\right)\le \mathbb{E}\left[\exp\left[ \lambda^2 \frac{\mathbf{\Phi }^2}{4}V^2\right]  \right]^{\frac{1}{2}} \exp\left[ \lambda^2\frac{\mathbf{\Phi }^2}{8} \mathbb{E}\left[V^2 \right]\right]\;,
\end{equation} 
where $V^2 \eqdef \sum_{i=1}^{n}U_i^2$. Using \eqref{NPHMM:assum:G} and by independence of the $(U_i)_{i\ge 0}$,
\begin{align*}
\exp\left(\psi_{S_n}\left(\frac{\lambda}{4}\right)\right)&\le \mathbb{E}\left[\exp\left[ \lambda^2 \frac{\mathbf{\Phi }^2}{4}\sum_{i=1}^{n} U_i^2\right]  \right]^{\frac{1}{2}} \exp\left[ \lambda^2\frac{\mathbf{\Phi }^2}{8} \sum_{i=1}^{n} \mathbb{E}[U_i^2]\right]\;,\\
&\le\prod_{i=1}^{n} \mathbb{E}\left[\exp\left[ \lambda^2 \frac{\mathbf{\Phi }^2}{4}  U_i^2\right] \right]^{\frac{1}{2}} \exp\left[ \lambda^2\frac{\mathbf{\Phi }^2}{8} \sum_{i=1}^{n} \mathbb{E}[U_i^2]\right]\eqsp.
\end{align*}
Thus,
\[
\psi_{S_n}(\lambda/4)\le \frac{1}{2} \sum_{i=1}^{n} \ln\left\{\mathbb{E}\left[\exp\left(\lambda^2 \frac{\mathbf{\Phi }^2}{4}  U_i^2\right)\right]\right\} + \lambda^2\frac{\mathbf{\Phi }^2}{8}\sum_{i=1}^{n} \mathbb{E}\left[U_i^2\right]\eqsp.
\]
Since for any $u>0$, $\ln(u)\le u-1$, this yields
\begin{equation*}
\psi_{S_n}(\lambda/4)\le  \frac{1}{2}\sum_{k=1}^{\infty}\frac{1}{k!}\left[ \lambda^2 \frac{\mathbf{\Phi }^2}{4} \right]^{k}  \sum_{i=1}^{n}\mathbb{E}\left[U_i^{2k}\right] + \lambda^2\frac{\mathbf{\Phi }^2}{8} \sum_{i=1}^{n} \mathbb{E}\left[U_i^2\right]\eqsp.
\end{equation*}
Then, by \eqref{NPHMM:assum:Ui},
\begin{equation*}
\psi_{S_n}(\lambda/4) \le  n\nu \left[ \lambda^2 \frac{\mathbf{\Phi }^2}{4} \right] \frac{1}{2}\sum_{k=0}^{\infty}\left[ \lambda^2 \frac{\mathbf{\Phi }^2}{4}c \right]^{k}+\left[ \lambda^2\frac{\mathbf{\Phi }^2}{8} \nu\right]\eqsp.
\end{equation*}
If $0<\lambda^2  \mathbf{\Phi }^2c/4<1$,
\begin{align*}
\psi_{S_n}(\lambda/4) &\le n\nu   \lambda^2 \frac{\mathbf{\Phi }^2}{8} \frac{1}{1-\lambda^2  \frac{\mathbf{\Phi }^2}{4} c} + n\nu\lambda^2\frac{\mathbf{\Phi }^2}{8}\eqsp,\\
&\le n\nu   \lambda^2 \frac{\mathbf{\Phi }^2}{4}  \frac{1}{1-\lambda^2  \frac{\mathbf{\Phi }^2}{4} c}\eqsp.
\end{align*}
Define $\nu' \eqdef 8 n\nu \mathbf{\Phi }^2 $  and $c' \eqdef 2\mathbf{\Phi }  \sqrt{c} $. Therefore,
\begin{equation} 
\psi_{S_n}(\lambda/4) \le \frac{\nu'(\lambda/4)^2}{2(1-c'(\lambda/4))}\eqsp.
\end{equation}
Hence, for all $0<\lambda <1/c'$,
\begin{equation}
\label{NPHMM:eq:Bernstein}
\psi_{S_n}(\lambda ) \le \frac{\nu'\lambda^2}{2(1-c'\lambda)}\;.
\end{equation}
By the Bernstein type inequality \eqref{NPHMM:eq:Bernstein}, \cite[Lemma 2.3]{massart:2007} gives, for any measurable set $A\subset\Omega$ with $\mathbb{P}(A) >0$,
\[
\mathbb{E}\left[S_n |A\right]\le \sqrt{2\nu' \ln\left(\frac{1}{\mathbb{P}(A)}\right)} + c' \ln\left(\frac{1}{\mathbb{P}(A)}\right)\eqsp. 
\]
Hence, by  \cite[Lemma 2.4]{massart:2007}, for any positive $x$, 
\[
\mathbb{P}\left[S_n \ge \sqrt{2\nu'x} + c'x \right] \le \mathrm{e}^{-x}\eqsp.
\]
\end{proof}
Proposition \ref{NPHMM:prop:ineg:maximale} below provides a control on the expectation of the empirical process. It introduces a $\beta$-mixing condition (see \cite{dedecker:2009}) which is weaker than the $\Phi$-mixing condition considered in Proposition~\ref{NPHMM:prop:concentration}. 
The $\beta$-mixing coefficient between two $\sigma$-fields $\mathcal{U},\mathcal{V}\subset \mathcal{E}$ is defined in \cite{dedecker:2009} by
\[
\beta(\mathcal{U},\mathcal{V}) \eqdef \frac{1}{2}\sup \sum_{(i,j)\in I\times J}\left|\mathbb{P}\left(U_i\cap V_j \right) -\mathbb{P}(U_i)\mathbb{P}(V_j)\right|\eqsp,
\]  
where the supremum is taken over all finite partitions $(U_i)_{i\in I}$ and $(V_j)_{j\in J}$ respectively $\mathcal{U}$ and $\mathcal{V}$ measurable.
The corresponding mixing coefficients $(\beta_i)_{i\ge 0}$ associated with a process $(X_k)_{k\ge0}$ satisfy  $\beta_i <\phi_i$ for all $i\ge 1$.

\begin{proposition}\label{NPHMM:prop:ineg:maximale}
Let $(Z_i)_{i\ge 0}$ be a stationary process taking values in a Polish space $\mathcal{Z}$ and let $\PPim$ be the distribution of $Z_0$. Assume that the sequence $(Z_i)_{i\ge0}$ is $\beta$-mixing and that
\[
\sum_{i=1}^{\infty}\beta_i <\infty\eqsp.
\]
Let $\mathcal{G}$ be a countable class of functions on $\mathcal{Z}$. Assume that there exist $r>1$ and $\delta>0$ such that for any $g\in \mathcal{G}$,
\[
||g||_{\mathrm{L}^{2r}(\PPim)} \eqdef \mathbb{E} \left[g(Z_0)^{2r}\right]^{1/2r} \le \delta\eqsp.
\]
Assume also that the bracketing function  satisfies
\[
\int_{0}^{1} \sqrt{H_{[]}(u,||\cdot||_{\mathrm{L}^{2r}(\PPim)},\mathcal{G})}\mathrm{d}u<\infty\eqsp.
\]
Then,
\[
\varphi(\delta) := \int_{0}^{\delta} \sqrt{H_{[]}(u,||\cdot||_{\mathrm{L}^{2r}(\PPim)},\mathcal{G})}\mathrm{d}u
\]
is finite and there exists a constant $A$ such that for $n$ big enough
\begin{equation}
\mathbb{E}\left[\sup_{g\in\mathcal{G}} |S_n(g)| \right]\le \sqrt{n}A\varphi(\delta)\eqsp,
\end{equation}
where, for all $g\in\mathcal{G}$, $S_n(g) =\sum_{i=0}^{n-1} g(Z_i) - n\mathbb{E}\left[g(Z_0)\right]$.
\end{proposition}
\begin{proof}
This is a direct application of the remark following \cite[Theorem 3]{doukhan:massart:rio:1995}.
\end{proof}

\section{Entropy of the class $\mathcal{G}_M$}
\label{NPHMM:sec:entropy}

\begin{lemma}
\label{NPHMM:prop:bracket:GM}
For any $q> 1$,  any $s>b\dimY/q$ and any even integer $\beta$,  provided that $\beta>s + b\dimY(1-1/q)$, there exists a constant $C$ such that for all $u>0$,
\begin{equation}\label{NPHMM:eq:brackets}
H_{[]}(u,||\cdot||_{\mathrm{L}^{2r}(\PPim)},\mathcal{G}_M) \le  C \left(\frac{M^{\upsilon(s+\beta+b\dimY/q)}}{u^{2r}} \right)^{b\dimY/s}\eqsp.
\end{equation}
\end{lemma}
 
\begin{proof}
By \cite[Lemma 7.26]{massart:2007}, for any probability densities $p_1$ and $p_2$ on $\mathbb{R}^{b\dimY}$,  
\begin{align*}
||g_{p_2} - g_{p_1}||^{2r}_{\mathrm{L}^{2r}(\PPim)} &\le C ||\sqrt{p_2}-\sqrt{p_1}||^2_{\mathrm{L}^{2}(\mathbb{R}^{b\dimY})}\eqsp.
\end{align*}
Since $||\sqrt{p_2}-\sqrt{p_1}||^2_{\mathrm{L}^{2}(\mathbb{R}^{b\dimY})} \le || p_2 - p_1 ||_{\mathrm{L}^{1}(\mathbb{R}^{b\dimY})}$, this yields, for any $u>0$,
\begin{equation}
\label{eq:HGHP}
H_{[]}(u,||\cdot||_{\mathrm{L}^{2r}(\PPim)},\mathcal{G}_M) \le H_{[]}\left(u^{2r}/C, || \cdot||_{\mathrm{L}^{1}(\mathbb{R}^{b\dimY})},\mathcal{P}_M\right)\eqsp, 
\end{equation}
where $\mathcal{P}_M \eqdef \left\{p_{f,\dens};\; \dens\in\mathcal{D}_b,\;f\in\mathcal{F} \; \mbox{and}\; I(f)\le M \right\}$. Thus, it remains to bound the entropy with bracketing of $\mathcal{P}_M$ associated with $|| \cdot||_{\mathrm{L}^{1}(\mathbb{R}^{b\dimY})} $ to control  the entropy with bracketing of $\mathcal{G}_M$ associated with $|| \cdot||_{\mathrm{L}^{2r}(\mathbb{P}_{\star})} $. For any $q > 1 $ and $s\ge 0$, define the Sobolev space on $\mathbb{R}^{b\ell}$:
\[
W^{s,q}\left(\mathbb{R}^{b\ell},\Rset\right) \eqdef \left\{h:\mathbb{R}^{b\ell} \to \mathbb{R}; \; \; D^{\alpha}h\in \mathrm{L}^{q}, \alpha\in\Nset^{b\ell}\; \mbox{and}\;0\leq|\alpha|\leq s\right\}\eqsp.
\]
For any $\beta>0$, let $\langle\cdot\rangle^{\beta}$ be the polynomial function on $\mathbb{R}^{b\dimY}$ given by   $\bfy \mapsto \langle \bfy\rangle^{\beta} \eqdef \left(1 + \norminf{\bfy}{}^{2} \right)^{\beta/2}$ and $W^{s,q}\left(\mathbb{R}^{b\ell},\langle\cdot\rangle^{\beta}\right)$ be the corresponding weighted Sobolev space:
\[
W^{s,q}\left(\mathbb{R}^{b\ell},\langle\cdot\rangle^{\beta}\right) \eqdef \left\{h:\mathbb{R}^{b\ell} \to \mathbb{R};\;\ \bfy \mapsto \langle\bfy\rangle^{\beta}h(\bfy) \in  W^{s,q}\left(\mathbb{R}^{b\ell},\Rset\right)\right\}\eqsp.
\] 
Lemma~\ref{NPHMM:lem:norm:sobo}  establishes that, for any $M\ge 1$, $q>1$, $s>b\dimY/q$ and even integer $\beta$, the normalized classes of functions $\mathcal{P}_{M}/M^{\upsilon(s+\beta+b\dimY/q)}$ are in the same bounded subspace of $W^{s,q}(\Rset^{b\dimY},\langle\bfy\rangle^{\beta})$. By \cite[Corollary 4]{nickl:potscher:2007}, for any $q> 1$, and any $s>b\dimY/q$, provided that $\beta>s + b\dimY(1-1/q)$, there exists a constant $C$ such that, for all $\epsilon>0$,
\begin{equation*} 
H_{[]}\left(\epsilon,\norminf{\cdot}{\rmL^{1}(\mathbb{R}^{b\dimY})},\mathcal{P}_{M}/M^{\upsilon(s+\beta+b\dimY/q)}\right) \le C \epsilon^{-b\dimY/s}\eqsp.
\end{equation*}
The proof is concluded by \eqref{eq:HGHP}.
\end{proof}

\begin{lemma}\label{NPHMM:lem:norm:sobo}
Assume that H\ref{assum:pen} holds for some $\upsilon>0$. Then, for any $q>1$, $s>b\dimY /q$ and any even $\beta>0$, there exists $C>0$ such that for any $f \in \mathcal{F}$ and any $\nu\in\mathcal{D}_b$,
\[
\norminf{\bfy\mapsto\langle \bfy\rangle^{\beta} p_{f,\nu}(\bfy)}{W^{s,q}(\mathbb{R}^{b\dimY},\Rset)} \le C(1\vee I(f)^\upsilon)^{s+\beta+b\dimY/q}\eqsp.
\]
\end{lemma}
\begin{proof}
Let $f$ be a function in $\mathcal{F}$, for any $\nu\in\mathcal{D}_b$, 
\[
\norminf{\bfy\mapsto\langle \bfy \rangle^{\beta} p_{f,\nu}(\bfy)}{W^{s,q}(\Rset^{b\dimY},\Rset)}^{q} = \sum\limits_{|\alpha| \le s} \norminf{D^{\alpha} \left(\langle \bfy\rangle^{\beta}p_{f,\nu}(\bfy)\right) }{\rmL^{q}}^{q}\eqsp.
\]
Applying the general Leibniz rule component by component yields, for any $\alpha \in \Nset^{b\dimY}$,
\begin{equation}\label{NPHMM:eq:derivative}
D^{\alpha} \left(\langle \bfy\rangle^{\beta}p_{f,\nu}(\bfy)\right)   = \sum\limits_{\alpha' \le \alpha} \prod_{j=1}^{b\dimY} {\alpha_{j}\choose\alpha_{j}'}  D^{\alpha'}(\langle \bfy\rangle^{\beta}) D^{\alpha-\alpha'}(p_{f,\nu}(\bfy))\eqsp.
\end{equation}
Then, Lemma~\ref{NPHMM:lem:norm:sobo} requires to control $\norminf{D^{\alpha^{(1)}}(\langle \bfy\rangle^{\beta})  D^{\alpha^{(2)}}(p_{f,\nu})  }{\rmL^{p'}}$ for any given $\alpha^{(1)}$ and $\alpha^{(2)}$ in $\Nset^{b\dimY}$. For any $\alpha$ in $\Nset^{b\dimY}$, there exists a polynomial function $P_{\alpha}$ with degree lower than $|\alpha|$ such that, for any $\bfy \in \Rset^{b\dimY}$,
\begin{equation}
\label{NPHMM:ed:sobo:derivative}
D^{\alpha} p_{f,\nu}(\bfy) = \int_{\bfx\in \mathbb{X}^b} P_{\alpha}(\mathbf{f}(\bfx) - \bfy)\exp\left\{-\frac{1}{2}\norm{\mathbf{f}(\bfx) - \bfy}^{2}\right\}  \nu(\bfx) \mu^{\otimes b}(\rmd \bfx)\eqsp. 
\end{equation}
Moreover, since $\beta$ is an even number, for any $\alpha \in \Nset^{b\dimY}$ such that $|\alpha | \le \beta$, $D^{\alpha} (\langle \bfy\rangle^{\beta})$ is a polynomial function denoted by $P_{\beta,\alpha}$ with degree lower than $\beta-|\alpha|$. In the case where $|\alpha | >  \beta$, $D^{\alpha}(\langle \bfy\rangle^{\beta})=0$. Define $\kappa(\upsilon,f)\eqdef 1\vee I(f)^\upsilon$. By H\ref{assum:pen}, there exists a constant $C>0$ such that, for any $\bfx\in\mathbb{X}^b$, $\norminf{\mathbf{f}(\bfx)}{}\leq  C I(f)^\upsilon  \leq  C \kappa(\upsilon,f)$. Since $P_{\alpha^{(2)}}$ and $P_{\beta,\alpha^{(1)}}$ are both polynomial functions, there exists a constant $C$ depending on $\alpha^{(1)},\alpha^{(2)}$ and $\beta$ such that, for any $\bfy\in\Rset^{b\dimY}$ and any $\bfx\in\mathbb{X}^{b}$,
\begin{equation*}
\left| P_{\beta,\alpha^{(1)}} (\bfy) P_{\alpha^{(2)}}(\mathbf{f}(\bfx) - \bfy) \right|
\le	 \1_{|\alpha^{(1)}| \le \beta} \left[ C (1+\norminf{\bfy}{})^{\beta-|\alpha^{(1)}|} \times\left( \kappa(\upsilon,f) +\norminf{\bfy}{}\right)^{|\alpha^{(2)}|}\right]\eqsp.		
\end{equation*}
Define the following subset of $\Rset^{b\dimY}$
\[
A_{f}\eqdef\left\{\bfy\in\Rset^{b\dimY};\; \norm{\bfy} \le  C\kappa(\upsilon,f) \right\} \eqsp.
\]
$\norm{\mathbf{f}(\bfx) - \bfy}$ can be lower bounded by 0 when $\bfy\in A_{f}$ and by $\norminf{\bfy}{} - C\kappa(\upsilon,f)$ when $\bfy\in A_{f}^{c}$. Therefore, uniformly in $\bfx \in \mathbb{X}^{b}$, 
\begin{equation*}
\exp\left\{-\norm{\bff(\bfx) - \bfy}^{2}/2\right\} \le \1_{A_{f}}(\bfy)
+ \1_{A_{f}^{c}}(\bfy)\exp\left\{-\left( C\kappa(\upsilon,f) -\norminf{y}{}\right)^{2}/2\right\}\eqsp.
\end{equation*}
Then, there exists a constant $C>0$, such that for any $q> 1$, 
\[
\norminf{D^{\alpha^{(1)}}(\langle \bfy\rangle^{\beta})  D^{\alpha^{(2)}}(p_{f,\nu})}{\rmL_{q}}^{q} \le\1_{|\alpha^{(1)}| \le \beta} \Bigg[ C\kappa(\upsilon,f)^{q|\alpha^{(2)}|}\left(I_{1} + I_{2}\right)\Bigg]\eqsp,
\]
where, 
\begin{align*}
I_{1}&\eqdef\int_{A_{f}} \left(1+\norminf{\bfy}{}\right)^{q(\beta-|\alpha^{(1)}|)}
\left(1+\frac{\norminf{\bfy}{}}{\kappa(\upsilon,f)}\right)^{q|\alpha^{(2)}|} \lambda^{\otimes b}(\rmd\bfy)\eqsp,\\
I_{2}&\eqdef \int_{A_{f}^{c}}  \left(1+\norminf{\bfy}{}\right)^{q(\beta-|\alpha^{(1)}|)}
\left(1+\frac{\norminf{\bfy}{}}{\kappa(\upsilon,f)}\right)^{q|\alpha^{(2)}|}\rme^{-q \left(C\kappa(\upsilon,f) -\norminf{y}{}\right)^{2}/2}\lambda^{\otimes b}(\rmd\bfy)\eqsp.
\end{align*}
By the change of variables $\bfz'=(\kappa(\upsilon,f))^{-1}\bfy$ in $I_{1}$ and $I_{2}$, there exists a constant $C$ such that
\begin{equation}\label{NPHMM:eq:dbl:deriv}
 \norminf{D^{\alpha^{(1)}}(\langle \bfy\rangle^{\beta})  D^{\alpha^{(2)}}(p_{f,\nu})}{\rmL_{q}}^{q} \le C\kappa(\upsilon,f)^{q(|\alpha^{(2)} |-|\alpha^{(1)}| + \beta) + b\dimY}\eqsp.
\end{equation}
Using \eqref{NPHMM:eq:dbl:deriv} in \eqref{NPHMM:eq:derivative}  with $\alpha^{(1)} = \alpha'$ and  $\alpha^{(2)} = \alpha-\alpha' $ for any $|\alpha|\le s$ and $\alpha'\le\alpha$ concludes the proof of Lemma~\ref{NPHMM:lem:norm:sobo}.
\end{proof}

\section{Proof of Lemma~\ref{NPHMM:lem:EM}}
\label{se:proof:EM}
The proof follows the same lines as the one for the usual EM algorithm. For all $0\le k\le n-1$, all $f\in\mathcal{F}$ and all $a\in\mathcal{A}$
\begin{align*}
\ln \left[p_{f,a}\left(\mathbf{Y}_{k}\right)\mathrm{e}^{-\lambda_nI(f)/n}\right]&=\ln \left[\int p_{f,a}\left(\bfx,\mathbf{Y}_{k}\right)\mathrm{e}^{-\lambda_nI(f)/n}\mu^{\otimes 2}(\rmd \bfx)\right]\eqsp,\\
&=\ln\left[\int p_{f,a}\left(\bfx,\mathbf{Y}_{k}\right)\mathrm{e}^{-\lambda_nI(f)/n}\frac{p_{\param{}{p},\widehat{a}^{p}}\left(\bfx|\mathbf{Y}_{k}\right)}{p_{\param{}{p},\widehat{a}^{p}}\left(\bfx|\mathbf{Y}_{k}\right)}\mu^{\otimes 2}(\rmd \bfx)\right]\eqsp,\\
&=\ln \left[\int p_{\param{}{p},\widehat{a}^{p}}\left(\bfx|\mathbf{Y}_{k}\right)\frac{p_{f,a}\left(\bfx,\mathbf{Y}_{k}\right)\mathrm{e}^{-\lambda_nI(f)/n}}{p_{\param{}{p},\widehat{a}^{p}}\left(\bfx|\mathbf{Y}_{k}\right)}\mu^{\otimes 2}(\rmd \bfx)\right]\eqsp,\\
&\ge\int p_{\param{}{p},\widehat{a}^{p}}\left(\bfx|\mathbf{Y}_{k}\right)\ln\hspace{-.1cm}\left[\frac{p_{f,a}\left(\bfx,\mathbf{Y}_{k}\right)\mathrm{e}^{-\lambda_nI(f)/n}}{p_{\param{}{p},\widehat{a}^{p}}\left(\bfx|\mathbf{Y}_{k}\right)}\right]\hspace{-.1cm}\mu^{\otimes 2}(\rmd \bfx)\eqsp,
\end{align*}
where the last inequality comes from the concavity of $x\mapsto \log x$. Then, 
\begin{multline*}
\ln \left[p_{f,a}\left(\mathbf{Y}_{k}\right)\mathrm{e}^{-\lambda_nI(f)/n}\right]-\ln \left[p_{\param{}{p},\widehat{a}^{p}}\left(\mathbf{Y}_{k}\right)\mathrm{e}^{-\lambda_nI(\param{}{p})/n}\right]\\
\geq \mathbb{E}_{\widehat{a}^p,\param{}{p}}\left[\ln p_{f,a}\left(\mathbf{X}_{k},\mathbf{Y}_{k}\right)-\ln p_{\param{}{p},\widehat{a}^{p}}\left(\mathbf{X}_{k},\mathbf{Y}_{k}\right)\middle| \mathbf{Y}_{k}\right] -\frac{\lambda_n}{n}\left(I(f)-I(\param{}{p})\right)\eqsp.
\end{multline*}
The proof is concluded by definition of $\widehat{a}^{p+1}$ and $\param{}{p+1}$.

\bibliographystyle{plain}
\bibliography{NPHMMbib}

\end{document}